\documentclass{singlecol-new}
\usepackage[cp1251]{inputenc}
\usepackage[english]{babel}
\usepackage{mathrsfs}
\usepackage{float}
\usepackage{setspace}
\usepackage{geometry}
\usepackage{graphicx}
\usepackage[section]{placeins}
\usepackage[hidelinks]{hyperref}
\usepackage{natbib}
\usepackage{graphicx,epstopdf}
\usepackage[caption=false]{subfig}
\usepackage{mathtext,enumerate,float}
\usepackage{ifpdf,xcolor}
\ifpdf
  \DeclareGraphicsExtensions{.eps,.pdf,.png,.jpg}
\else
  \DeclareGraphicsExtensions{.eps}
\fi

\usepackage{verbatim}

\theoremstyle{TH}{
\newtheorem{lemma}{Lemma}
\newtheorem{theorem}[lemma]{Theorem}
\newtheorem{corrolary}[lemma]{Corrolary}

}

\theoremstyle{THrm}{
\newtheorem{definition}{Definition}[section]

\newtheorem{remark}{Remark}

}

\theoremstyle{THhit}{

}

\makeatletter
\def\theequation{\arabic{equation}}

\JOURNALNAME{\TEN{\it Int. J. Computing Science and Mathematics,
Vol. \theVOL, No. \theISSUE, \thePUBYEAR\hfill\thepage}}%

\def\BottomCatch{%
\vskip -10pt
\thispagestyle{empty}%
\begin{table}[b]%
\NINE\begin{tabular*}{\textwidth}{@{\extracolsep{\fill}}lcr@{}}%
\\[-12pt]
Copyright \copyright\ 20XX Inderscience Enterprises Ltd. & &%
\end{tabular*}%
\vskip -30pt%
\end{table}%
} \makeatother

\begin{document}%


\setlength{\textfloatsep}{5pt plus 1.0pt minus 1.0pt}
\setlength{\floatsep}{4pt plus 1.0pt minus 1.0pt}
\setlength{\intextsep}{5pt plus 1.0pt minus 1.0pt}

\newcommand{\Rn}{{\mathbb R}^n}

\setcounter{page}{1}

\LRH{M.S. Filipkovska}

\RRH{Two combined methods for implicit semilinear differential equations}

\VOL{x}

\ISSUE{x}

\PUBYEAR{xxxx}

\BottomCatch

\CLline

\subtitle{}

\title{Two combined methods for the global solution of implicit semilinear differential equations with the use of spectral projectors and Taylor expansions}

\authorA{Maria S. Filipkovska} 
\affA{B. Verkin Institute for Low Temperature Physics and Engineering of the National Academy of Sciences of Ukraine,\\
47 Nauky Ave., Kharkiv, 61103, Ukraine, \\[2pt]
and\\[2pt]
V.N. Karazin Kharkiv National University,\\
4 Svobody Sq., Kharkiv, 61022,  Ukraine,\\
E-mail: filipkovskaya@ilt.kharkov.ua\\
E-mail: maria.filipk@gmail.com}

\begin{abstract}
Two combined numerical methods for solving implicit semilinear differential equations are obtained and their convergence is proved.   The comparative analysis of these methods is carried out and conclusions about the effectiveness of their application in various situations are made.  In comparison with other known methods, the obtained methods require weaker restrictions for the nonlinear part of the equation.  Also, the obtained methods enable to compute approximate solutions of the equations on any given time interval and, therefore, enable to carry out the numerical analysis of global dynamics of the corresponding mathematical models. The examples demonstrating the capabilities of the developed methods are provided.  To construct the methods we use the spectral projectors, Taylor expansions and finite differences.   Since the used spectral projectors can be easily computed, to apply the methods it is not necessary to carry out additional analytical transformations.
\end{abstract}

\KEYWORD{implicit differential equation; differential-algebraic equation; combined method; regular pencil; spectral projector; global dynamics.}

\REF{to this paper should be made as follows: Filipkovska, M.S. (xxxx) `Two combined methods for the global solution of implicit semilinear differential equations with the use of spectral projectors and Taylor expansions', {\it Int. J. Computing Science and Mathematics}, Vol.~x, No.~x,
pp.xxx--xxx.}


\maketitle

 \section{Introduction}\label{Intro}

Consider an implicit semilinear differential equation
\begin{equation}\label{DAE}
 \frac{d}{dt}[Ax]+Bx=f(t,x)
\end{equation}
with the initial condition
\begin{equation}\label{ini}
    x(t_0)=x_0,
\end{equation}
where $t,\, t_0 \ge 0$, $x,\, x_0 \in \Rn$, $f\in C([0,\infty)\times \Rn,\Rn)$, and $A,\, B\colon \Rn \to \Rn$~are linear operators (or the corresponding $n\times n$ matrices) which may be degenerate (noninvertible).   An equation of the type~\eqref{DAE} with a degenerate operator at the derivative is called a \emph{differential-algebraic equation}  (\emph{DAE}),  since a system of differential and algebraic equations corresponds to it.   In many papers and books, a DAE or a differential-algebraic system is called an algebraic-differential equation or system, degenerate differential equation, singular system, descriptor system, etc.  Differential-algebraic equations (DAEs) in descriptor form, i.e., expressed in terms of descriptor variables, or descriptor systems are used in modelling power, control, biological, chemical, and mechanical multibody systems (see, e.g., \cite{BrenanCP},  \cite{FoxJenZom}, \cite{RabierRh}, \cite{KunkelMehrmann}, \cite{EngwerdaSalmW}, \cite{Riaza}).  The presence of a degenerate operator at the derivative in the DAE means the presence of algebraic constraints, namely, the graphs of the solutions must lie in some manifold and the initial points $(t_0,x_0)$ must also belong to this manifold.  The initial value $x_0$ for the problem \eqref{DAE}, \eqref{ini} with a degenerate operator $A$ must be chosen so that the initial point $(t_0,x_0)$ belongs to the manifold
\begin{equation}\label{L_0}
L_0 = \{(t,x) \in [0,\infty) \times \Rn \mid Q_2[Bx-f(t,x)] = 0\},
\end{equation}
where $Q_2$ is the spectral projector considered in Section~\ref{IndexMeth}.  If the operator $A$ is nondegenerate (invertible), then the equation \eqref{DAE} can be written as an ordinary differential equation (ODE) and $Q_2= 0$. The DAE~\eqref{DAE} is called \emph{semilinear} and is often written in the form ${\displaystyle\frac{d}{dt}[Ax]=g(t,x)}$.  There are several reasons why we consider a semilinear DAE in the form~\eqref{DAE}. Firstly, DAEs of this form are used to describe mathematical models in radio electronics, economics, control theory, mechanics of multilink mechanisms, chemical kinetics, ecology and other fields (see references in Section~\ref{Appl}).  Secondly, in the equation~\eqref{DAE}, not only the operator $A$ but also the operator $B$ may be degenerate, and the influence of the linear part $\displaystyle\frac{d}{dt}[Ax]+Bx$ of the equation is determined by the properties of the pencil $\lambda A+B$.   It is assumed that $\lambda A+B$ is a regular pencil of index not higher than~1  (see Section~\ref{IndexMeth}). Then there exist the spectral projectors $P_1$, $P_2$, $Q_1$, $Q_2$ (\cite{Rut}) (see Section~\ref{IndexMeth}) which are used in the development of the numerical methods and in the proof of theorems.    Note the following:  we do not require that the equation~\eqref{DAE} be a regular DAE of index~1 for all  $(t,x)\in [0,\infty )\times \Rn$ (or $(t,x)\in L_0$), i.e., that the pencil $\displaystyle\lambda A+B-\frac{\partial f}{\partial x}(t,x)$  be a regular pencil of index~1 for all $(t,x)$.      This requirement is one of the conditions for the global solvability of a nonlinear DAE in the theorem [\cite{LamourMarzTisch}, Theorem~6.7].   In \cite{AscherPetz}, \cite{BrenanCP}, \cite{KunkelMehrmann}, \cite{HairerLR}, \cite{HairerW}, and \cite{Knorrenschild}, restrictions similar to the above requirement of index~1 for a regular DAE are used locally to prove the local solvability of DAEs. Various notions of an index for a regular DAE and the relationship between them are discussed in [\cite{Fil.MPhAG}, Remark~2.1].
In what follows, for the sake of generality, the equation~\eqref{DAE} with an arbitrary (not necessarily degenerate) linear operator  $A\colon \Rn \to \Rn$  will be called a semilinear DAE.

A function $x(t)$ is called a \emph{solution} of the initial value problem \eqref{DAE}, \eqref{ini} on some interval $[t_0,t_1)$, $t_1\le \infty $, if $x \in C([t_0,t_1),\, \Rn)$, $Ax \in C^1([t_0,t_1),\Rn)$, $x$~satisfies the equation \eqref{DAE} on $[t_0,t_1)$ and the initial condition \eqref{ini}.    A solution $x(t)$ of the initial value problem \eqref{DAE}, \eqref{ini} is called \emph{global} if it exists on the interval $[t_0,\infty )$.

It should be noted that in the case of the degenerate operator $A$ a solution of a semilinear DAE of the form  $\displaystyle A\frac{d}{dt}x+Bx=f(t,x)$  must be smoother (the solution is continuously differentiable) than a solution of a semilinear DAE of the form \eqref{DAE}.

There are various ways of finding approximate solutions of DAEs.
In most works, the main idea is the reduction of a DAE to an ODE or the replacement of a DAE by a stiff ODE for the further application of the known methods for solving ODEs, as well as the use of these methods directly for solving DAEs.
In \cite{GearPetz84}, various algorithms for reducing regular linear DAEs to ODEs have been presented and references to earlier works have been given.  In addition, the BDF (backward differentiation formulas) method for solving a regular nonlinear DAE having uniform index 1 has been proposed (see also \cite{BrenanCP}, \cite{AscherPetz}).  The $\varepsilon$-embedding method, the BDF method and general linear multi-step methods (see \cite{AscherPetz}, \cite{BrenanCP}, \cite{KunkelMehrmann}, \cite{Knorrenschild}, \cite{HairerW}, \cite{HairerLR}) are applied to semi-explicit DAEs
$\dot{y}= f(t,y,z)$,  $0=g(t,y,z)$ of index 1  (this DAE has index 1 for all $t$, $y$, $z$ such that $\left[\frac{\partial g}{\partial z}(t,y,z)\right]^{-1}$ exists and is bounded).  The $\varepsilon$-embedding method is as follows: The corresponding stiff system of ODEs  $\dot{y}= f(t,y,z)$, $\varepsilon  \dot{z}=g(t,y,z)$, $\varepsilon \to 0$, is considered, then Runge-Kutta or other suitable methods are applied to the stiff system and $\varepsilon = 0$ is put in the resulting formulas.  The semi-explicit DAE corresponding to the stiff ODE system is called reduced. Note that the solution of a perturbed (stiff) ODE system of the form $\dot{y}=f(t,y,z,\varepsilon)$, $\varepsilon\dot{z}=g(t,y,z,\varepsilon)$, where $\varepsilon >0$ is a small parameter, in general does not approach the solution of the reduced DAE (obtained by setting $\varepsilon=0$)\; $\dot{y}= f(t,y,z)=f(t,y,z,0)$, $0=g(t,y,z)=g(t,y,z,0)$ (cf. \cite{Knorrenschild}, \cite{BrenanCP}). Under certain conditions, it is possible to construct a stiff ODE system whose solutions converge in some sense towards the solution of the reduced DAE as $\varepsilon\to 0$ (see, e.g., \cite{Knorrenschild}). Also, under certain conditions the solution of the reduced DAE is a good approximation to the solution of the corresponding stiff ODE system. The considered perturbed or stiff ODE system is of interest by themselves because it arise in many applications. Such systems may contain slow and fast variables (motions) and describe, in particular, the motions of rigid bodies about the center of mass under the action of the moments of forces of various physical nature (\cite{CheAkLe}, \cite{BogSob}).     To study these systems, various asymptotic methods,  e.g., the averaging method allowing one to find approximate solutions of the perturbed system on an asymptotically large interval of time (\cite{CheAkLe}; \cite{VolMorg}) and the method of integral manifolds for the asymptotic separation of fast and slow motions (\cite{BogSob}, \cite{KobrSob}), are used.    For a regular nonlinear DAE of index 1 (this condition for the semilinear DAE was discussed above) (\cite{AscherPetz}, \cite{BrenanCP}, \cite{LamourMarzTisch}) and for a regular quasilinear DAE of the form $C(y)\dot{y}=f(y)$ with constraints providing the local solvability (\cite{HairerW}), the application of the BDF methods, the Runge-Kutta methods, the backward Euler method (the implicit Euler method) and the general linear multi-step method has been considered.   In \cite{KunkelMehrmann}, the BDF and linear multi-step methods are used for solving regular linear DAEs with constant matrix coefficients and it is discussed why similar results cannot hold for linear DAE with variable coefficients.   The collocation Runge–Kutta method and the BDF method for solving a regular strangeness-free DAE (with the strangeness index 0) of the form  $F_1(t,x,\dot{x})=0$, $F_2(t,x)=0$ are also proposed in \cite{KunkelMehrmann}.   In \cite{BoyarDanLCh}, an analog of the Euler method is applied to the nonlinear system $f(\dot{x},x,t)=0$  when special conditions are fulfilled and a DAE of the form $A(t)\dot{x}+\Phi(x,t)=0$ is considered as a particular case.  The combination of the simple iteration method and the explicit Euler method is used for solving a degenerate semilinear integro-differential equation in \cite{Piven}.  In \cite{Kylikov}, two combined methods such as the implicit Euler method in combination with the simple iteration method  and the implicit Adams method in combination with the Newton method are proposed for solving the autonomous semi-explicit DAE  $\dot{y}= f(y,z)$, $z= g(y,z)$.

Most of the mentioned methods for solving semilinear and nonlinear DAEs can be correctly applied only on a sufficiently small (local) interval of time and the calculation of the allowable length of this interval is a separate problem.  This is due to the fact that the existence of exact and approximate solutions is proved only on a sufficiently small time interval (\cite{HairerW}, \cite{HairerLR}, \cite{Knorrenschild}, \cite{AscherPetz}, \cite{BrenanCP}, \cite{KunkelMehrmann}, \cite{BoyarDanLCh}, \cite{LamourMarzTisch}).   However, for the analysis of the global dynamics of mathematical models, it is important to be able to investigate the behavior of the solution on an arbitrary (arbitrarily large) time interval.

The conditions, used in \cite{GearPetz84}, \cite{Kylikov}, \cite{LamourMarzTisch}, and \cite{Piven} to prove the existence of exact solutions and the convergence of the numerical methods on an arbitrary time interval, are too restrictive for certain classes of mathematical models. Namely, the global Lipschitz condition and similar conditions (for example, the global condition of contractivity and the condition of global index~1 for a regular DAE) are used. These conditions are not fulfilled for mathematical models of electrical circuits with certain nonlinear parameters (for example, in the form of power functions mentioned in Section~\ref{Appl}) on the time interval $[t_0,T]$, which is not sufficiently small, and, obviously, on the interval $[t_0,\infty )$.  The intervals $[t_0,T]$ and $[t_0,\infty )$ are considered when finding numerical solutions and when studying the existence, uniqueness and Lagrange stability of exact solutions, respectively.   In general, various types of implicit differential equations with nonlinear functions which may not satisfy the global Lipschitz condition and similar conditions arise in many applications (for example, various classes of implicit stochastic differential equations with non-globally Lipschitz functions) (\cite{BenchaabaneSakth}).

For the application of many methods, additional analytical transformations and constructions are required. To use the algorithms and methods from \cite{GearPetz84}, it is necessary that there exist the projection matrices or mappings which transform the DAE to a canonical or semicanonical form.  In some works, conditions for the existence of operators allowing one to reduce a semilinear DAE to a semi-explicit DAE are given, however, their construction requires additional transformations and computations (cf. \cite{BoyarDanLCh}, \cite{LamourMarzTisch}).   It should be noted that one can reduce the semilinear DAE~\eqref{DAE} to the equivalent semi-explicit DAE, using the spectral projectors discussed in Section~\ref{IndexMeth}.  Further, in order to apply the aforementioned $\varepsilon$-embedding method, it is necessary to reduce a semi-explicit DAE to the corresponding stiff system of ODEs that requires a sufficient smoothness of the nonlinear function in the algebraic part of the DAE (the continuity of the function is not enough here).  The methods considered above require that the nonlinear functions appearing in the equation be at least continuously differentiable in all arguments, and most of the methods require the higher smoothness of the nonlinear functions.

In this paper, we obtain numerical methods having the following \emph{advantages} in comparison with other known methods for solving equations of the type~\eqref{DAE}:

1.~Applying the obtained methods, it is possible to compute approximate solutions on any given time interval $[t_0,T]$.   The theorems on the existence and uniqueness of global solutions (see  Section~\ref{GlobSolv}) and on the convergence of the numerical methods  (see  Section~\ref{NumMs}) ensure the correctness and convergence of the methods.  The results of the theoretical research of the global dynamics of mathematical model considered in Section~\ref{Appl} are consistent with the results of the analysis of numerical solutions.

2.~The obtained methods require weaker restrictions for the nonlinear part of the equation.  To prove the existence and uniqueness of global solutions and to prove the convergence of the methods, the restrictions of the type of the global Lipschitz condition are not used (we discussed these conditions above). Moreover, the methods require the less smoothness the nonlinear part of the equation than other known methods.  They are applicable to the DAEs with the continuous nonlinear part which may not be continuously differentiable in $t$ (see Remarks~\ref{remNum-meth},~\ref{remModNum-meth}). This is important for applications, since such equations arise in various practical problems.  In particular, the functions of currents and voltages in electric circuits may not be continuously differentiable (or be piecewise continuously differentiable) or may be approximated by functions that are not continuously differentiable.  As examples, nonsinusoidal currents and voltages of the ``sawtooth'', ``triangular'' and ``rectangular'' shapes (see, e.g., \cite{EricksonMaks}) can be considered.  In Section~\ref{Appl}, the example of the numerical solution for an electrical circuit with the voltage of the triangular shape is given.

3.~The spectral projectors, with the help of which the DAE \eqref{DAE} is reduced to the equivalent system of a purely differential equation and a purely algebraic equation  (to the semi-explicit form), can be constructively determined by the formulas \eqref{Proj.1} (\cite{RutVlas})  and be numerically found using \eqref{ProjRes}.   The possibility to easily compute the projectors on a computer, using \eqref{ProjRes}, enables to numerically solve the DAE directly in the form \eqref{DAE}, i.e., additional analytical transformations are not required for the application of the developed numerical methods.

To construct the numerical methods, the (differential and algebraic) equations of the system equivalent to the DAE~\eqref{DAE} (see item~3 above) are approximated by using Taylor expansions and finite differences.   As a result, the combined numerical methods are obtained.   In method~1 (the method \eqref{met1}--\eqref{met4}), we apply the explicit Euler method to the differential equation, that is, the derivative is approximated by a forward difference.   In method~2 (the method \eqref{Impmet1}--\eqref{Impmet4}), the derivative is approximated by a centered difference, which leads to a certain modification of the Euler method.  Therefore, method~2 will also be called modified method~1.  In the algebraic equation (for both methods), the Taylor expansion of the nonlinear function in one of the components of the phase variable is used, which gives a method similar to the Newton method with respect to this component. This technique of the expansion allows us to weaken the requirements for the nonlinear function and to apply the obtained methods even for the DAE with the nonlinear part continuous in $t$ (taking into account Remarks~\ref{remNum-meth},~\ref{remModNum-meth}).   The consistency condition $(t_0,x_0)\in L_0$ for the initial values $t_0$,~$x_0$ ensures the best choice of the initial value for the method applied to the algebraic equation.  Methods 1 and 2 have the first and second orders of accuracy, respectively.  It is clear that the approximations of a higher order of accuracy are used for method~2 and, accordingly, the higher smoothness of the nonlinear part of the equation is required (see Theorems~\ref{ThNum-meth},~\ref{ThModNum-meth}).

\emph{The paper has the following structure.}   In Section~\ref{IndexMeth}, we consider the restriction on the operator coefficients of the equation \eqref{DAE} (on the operator pencil) and give the corresponding definition of a regular pencil of index not higher than~1; also, we consider the method of spectral projectors for the reduction of the semilinear DAE to an equivalent semi-explicit form.   In Section~\ref{GlobSolv}, the theorems on the existence and uniqueness of global solutions are presented.
In Section~\ref{NumMs}, the two combined numerical methods for solving the semilinear DAE are obtained and their convergence is proved (note that the method from Subsection~\ref{Num-meth} was proposed by the author in \cite{Filnummeth} without the theorem on its convergence), as well as the important remarks on the convergence of the methods, when weakening the smoothness requirements for the nonlinear function, are given.   In Section~\ref{CompareMeth}, the comparative analysis of these methods is carried out and conclusions about the effectiveness of their application in various situations are made. In Section~\ref{Appl}, we provide the examples demonstrating the capabilities of the developed methods and information on applied problems in which the semilinear DAEs arise.

The following notation is used in this paper: $L(X,Y)$ is the space of continuous linear operators from $X$ into $Y$, $L(X,X)=L(X)$; $\left. A\right|_X$ is the restriction of the operator $A$ to $X$; $I$ is an identity operator (matrix); $x^T$ is the transpose of $x$. Sometimes, in the paper, a function $f$ is denoted by the same symbol $f(x)$ as its value at the point $x$ in order to explicitly indicate its argument (or arguments), but it will be clear from the context what exactly is meant.

\section{Index of the regular pencil $\lambda A+B$ and spectral projectors}\label{IndexMeth}

Consider the initial value problem \eqref{DAE}, \eqref{ini}:
\begin{equation*}
 \frac{d}{dt}[Ax]+Bx=f(t,x), \qquad x(t_0)=x_0.
\end{equation*}

It is assumed that the pencil $\lambda A+B$, where $\lambda$ is a complex parameter, is \emph{regular}, i.e.,  there exists a number $\lambda_0$ such that $\det(\lambda_0 A+B)\not = 0$.  Further, we assume that there exist constants $C_1,\, C_2 >0$ such that
\begin{equation}\label{index1}
\left\|(\lambda A+B)^{-1}\right\| \le C_1\quad \text{ for all }\; |\lambda|\ge C_2.
\end{equation}
The condition \eqref{index1} (\cite{RutVlas}) means that either the point $\mu = 0$ is a simple pole of the resolvent $(A+ \mu B)^{-1}$ (this is equivalent to the fact that $\lambda = \infty$ is a removable singular point of the resolvent $(\lambda A+B)^{-1}$), or $\mu = 0$ is a regular point of the pencil $A+ \mu B$ (i.e., there exists a resolvent $(A+\mu B)^{-1}$ at the point $\mu =0$ and, hence, the operator $A$ is nondegenerate).

If $A$ is degenerate and the point $\mu = 0$  is a simple pole of the resolvent $(A+ \mu B)^{-1}$, i.e., \eqref{index1} is fulfilled, then we will say that $\lambda A+ B$ is a regular \emph{pencil of index~1}.  If $A$ is nondegenerate,  i.e., $\mu = 0$ is a regular point of the pencil $A+ \mu B$, then we will say that $\lambda A+ B$ is a regular \emph{pencil of index~0}.  Thus, if $\lambda A+B$~is a regular pencil and \eqref{index1} is fulfilled, then we will say that $\lambda A+B$~is a regular \emph{pencil of index not higher than 1} (i.e., of index 0 or 1).

In the general case, according to [\cite{Vlasenko1}, Section~6.2], the maximum length of the chain of an eigenvector and adjoint vectors of the matrix pencil $A+\mu B$ at the point $\mu = 0$ is called the index of the matrix pencil $\lambda A+B$.   Various  notions of an index of the pencil, an index of a DAE and their relationship with the mentioned notion of the pencil of index 1 are considered in [\cite{Fil.MPhAG},~Remark~2.1].

For the regular pencil $\lambda A+B$ satisfying \eqref{index1} there are the two pairs of mutually complementary spectral projectors [\cite{RutVlas}, Sections~2,~6]
\begin{equation}\label{Proj.1}
 \begin{array}{l}
 \displaystyle P_1 =\frac{1}{2\pi i}\, \oint\limits_{|\lambda |=C_2}(\lambda A+B)^{-1}\, A\, d\lambda,\quad  P_2 =I-P_1, \\
 \displaystyle Q_1 =\frac{1}{2\pi i}\, \oint\limits_{|\lambda |=C_2}A\, (\lambda A+B)^{-1}\,d\lambda,\quad  Q_2 =I-Q_1,
 \end{array}
\end{equation}
which decompose the space $\Rn$ into direct sums of subspaces
\begin{equation}\label{Proj.2}
 \Rn =X_1 \dot{+}X_2,\quad  \Rn=Y_1 \dot{+} Y_2,\quad  X_j =P_j \Rn,\quad Y_j =Q_j \Rn,\quad  j=1,2,
\end{equation}
such that the operators $A$, $B$ map $X_j$ to $Y_j$ and the induced operators ${A_j=\left. A\right|_{X_j}\colon X_j \to Y_j}$, ${B_j=\left. B\right|_{X_j}\colon X_j \to Y_j}$ (${j=1,2}$) are such that $A_2=0$ and  there exist the inverse operators $A_1^{-1} \in L(Y_1,X_1)$, $B_2^{-1} \in L(Y_2,X_2)$.   The projectors are real   (since $A$ and $B$ are real) and have the properties  $AP_1 =Q_1 A=A$, $AP_2 =Q_2 A=0$, $BP_j =Q_j B$, $j=1,2$.

Using the spectral projectors, we can also obtain the auxiliary operator $G\in L(\Rn)$ (\cite{RutVlas})
\begin{equation}\label{Proj.G}
G=A +BP_2 = A+Q_2 B,\quad GX_j =Y_j,\quad j=1,2,
\end{equation}
which has the inverse operator $G^{-1}\in L(\Rn)$ with the properties $G^{-1} AP_1 =P_1$, $G^{-1} BP_2 =P_2$, $AG^{-1} Q_1 =Q_1$, $BG^{-1} Q_2 =Q_2$.

Obviously, the projectors \eqref{Proj.1} can be calculated by using residues:
\begin{equation}\label{ProjRes}
\begin{array}{l}
 \displaystyle P_1 = \mathop{Res }\limits_{\mu =0}\left(\frac{(A+\mu B)^{-1} A}{\mu} \right),\quad  P_2 =I-P_1,  \\
 \displaystyle Q_1 =\mathop{Res }\limits_{\mu =0}\left(\frac{A(A+\mu B)^{-1}}{\mu} \right),\quad  Q_2 =I-Q_1.
 \end{array}
\end{equation}

With respect to the decomposition \eqref{Proj.2} any vector $x\in \Rn$ can be uniquely represented in the form
\begin{equation}\label{xr}
x=x_{p_1} +x_{p_2},\qquad x_{p_1} =P_1 x\in X_1,\quad x_{p_2} =P_2 x\in X_2.
\end{equation}

Applying the spectral projectors $Q_1$, $Q_2$ to the DAE \eqref{DAE} and taking into account their properties, we obtain the equivalent system
\begin{equation}\label{systDAE1}
 \begin{array}{l}
 \displaystyle \frac{d}{dt} [AP_1 x]+BP_1 x = Q_1 f(t,x), \smallskip\\
  BP_2 x = Q_2 f(t,x).
 \end{array}
\end{equation}
In the equations of the system we restrict the operators to the subspaces $X_1$, $X_2$ from \eqref{Proj.2}. Taking into account the invertibility of the induced operators $A_1$, $B_2$ and the representation \eqref{xr}, the system \eqref{systDAE1} can be rewritten in the form
\begin{align*}
 & \frac{d}{dt} x_{p_1}+ A_1^{-1} B_1 x_{p_1} =  A_1^{-1} Q_1 f(t, x_{p_1} +x_{p_2}), \\
 &  x_{p_2} =B_2^{-1} Q_2 f(t, x_{p_1} +x_{p_2}).
\end{align*}
Thus, the spectral projectors allow one to reduce the original semilinear DAE \eqref{DAE} to the equivalent system of purely differential and purely algebraic equations (to the semi-explicit form).

Since the projectors are easily computed with the help of \eqref{ProjRes} on a computer, and, consequently, the operator $G$ \eqref{Proj.G} are easily computed, to construct the numerical methods it is more convenient to use $G^{-1}\in L(\Rn)$. Using $G^{-1}$, we can write the system \eqref{systDAE1}  (equivalent to the DAE \eqref{DAE}) as
\begin{equation}\label{prThreg1}
\begin{array}{l}
 \displaystyle \frac{d}{dt} [P_1 x]+G^{-1} BP_1 x  = G^{-1} Q_1 f(t,P_1x+P_2x), \smallskip\\
 P_2 x  =  G^{-1} Q_2 f(t,P_1x+P_2x).
\end{array}
\end{equation}
Thus, the possibility to easily compute the spectral projectors enables to numerically solve the DAE directly in the form \eqref{DAE}, i.e., to apply the developed numerical methods it is not necessary to carry out additional analytical transformations.

\section{The existence and uniqueness of global solutions}\label{GlobSolv}

Recall that a function $x \in C([t_0,t_1),\, \Rn)$ is a solution of the initial value problem \eqref{DAE}, \eqref{ini} on the interval $[t_0,t_1)$\, ($t_1\le \infty $) if $Ax \in C^1([t_0,t_1),\Rn)$, $x$ ~satisfies the equation \eqref{DAE} on $[t_0,t_1)$ and the initial condition \eqref{ini}.  The solution is global if it exists on the whole interval $[t_0,\infty )$.  To formulate the theorems given below, we need the following definition.
\begin{definition}
(\cite{RF1}, \cite{Fil.MPhAG})\, An operator function (a mapping) $\Phi \colon D\to L(W,Z)$, where $W$,~$Z$ are $s$-dimensional linear spaces and $D\subset W$, is called \textit{basis invertible} on an interval $[\Hat w, \Tilde w]$, where $\Hat w,\, \Tilde w\in D$,  if for some additive resolution of the identity $\{\Theta _k\}_{k=1}^s$~in~the space $Z$  (see [\cite{Fil.MPhAG}, Definition~2.2] or [\cite{RF1}, Definition~2]) and for any set of vectors $\{w^k\}_{k=1}^s \subset [\Hat w, \Tilde w]$  the operator $\Lambda=\sum\limits_{k=1}^s \Theta _k \Phi (w^k) \in L(W,Z)$ has the inverse $\Lambda ^{-1} \in L(Z,W)$.
\end{definition}
In this paper, we use for convenience the term ``interval $[\Hat w, \Tilde w]$'' instead of the term ``convex hull $conv\{\Hat w, \Tilde w\}$ of vectors $\Hat w,\, \Tilde w$'' used in \cite{RF1} and \cite{Fil.MPhAG},  taking into account that $conv\{\Hat w, \Tilde w\}=[\Hat w, \Tilde w]=\{\alpha \Tilde w + (1-\alpha)\Hat w \mid \alpha \in [0,1]\}$.

Note that the property of basis invertibility does not depend on the choice of an additive resolution of the identity in $Z$.
Also note that if the operator function $\Phi $ is basis invertible on $[\Hat w, \Tilde w]$, then it is invertible on $[\Hat w, \Tilde w]$, i.e., for each point $w^*\in [\Hat w, \Tilde w]$  its image $\Phi (w^*)$ under the mapping $\Phi $ is an invertible continuous linear operator from $W$ into $Z$.  The converse statement is not true unless the spaces $W$, $Z$ are one-dimensional (see [\cite{Fil.MPhAG}, Example~2.1]).

Below we use the projectors $P_j$, $Q_j$ and the subspaces $X_j$, $Y_j$, $j=1,2$, discussed in Section~\ref{IndexMeth}.  Recall that $L_0 = \{(t,x) \in [0,\infty) \times \Rn \mid Q_2[Bx-f(t,x)]=0\}$ \eqref{L_0}.
\begin{theorem}\label{ThRF1}
Let $f\in C([0,\infty)\times \Rn, \Rn)$,\, $\displaystyle\frac{\partial f}{\partial x}  \in C([0,\infty)\times \Rn, L(\Rn))$, and $\lambda A+B$ be a regular pencil of index not higher than~1.  Assume that for any $t\ge 0$ and any $x_{p_1}\in X_1$ there exists $x_{p_2} \in X_2$ such that
\begin{equation}\label{soglreg2}
 (t,x_{p_1}+x_{p_2})\in L_0,
\end{equation}
and for any $x^i_{p_2} \in X_2$ such that $(t_*,x_{p_1}^*+x^i_{p_2}) \in L_0$, $i=1,2$, the operator function
\begin{equation}\label{funcPhi}
\Phi \colon X_2 \to L(X_2,Y_2),\quad \Phi (x_{p_2})=\left[\frac{\partial Q_2 f}{\partial x} (t_*,x_{p_1}^* +x_{p_2})-B\right] P_2,
\end{equation}
is basis invertible on $[x^1_{p_2},\,x^2_{p_2}]$. Assume that there exists a self-adjoint positive operator $H\in L(X_1)$ and for each $T>0$ there exists a number $R_T >0$ such that
\begin{equation}\label{ineqreg}
 (H P_1 x,G^{-1}Q_1f(t,x))\le 0\quad \text{ for all }\; (t,x)\in L_0\; \text{ such that }\; 0\le t\le T,\; \|P_1 x\| \ge R_T.
\end{equation}
Then for each initial point $(t_0,x_0)\in L_0$ there exists a unique global solution $x(t)$ of the initial value problem \eqref{DAE}, \eqref{ini}.
\end{theorem}

\proof{Proof}   The proof of the theorem is carried out in the same way as the proof of Theorem~1 from \cite{RF1}.  If the DAE \eqref{DAE} has the regular pencil of index~0 ($A$ is nondegenerate), then it can be reduced to an ordinary differential equation. In this case, $Q_2=P_2=0$, $Q_1=P_1 = I$ and  \eqref{soglreg2}, \eqref{funcPhi} are absent (\cite{Fil}).
\endproof
\begin{corrolary}\label{conseq-ThRF1}
Assume that in Theorem~\textup{\ref{ThRF1}} the projection $Q_1 f$ admits the representation
\begin{equation}\label{consRF1-1}
Q_1 f(t,x)=S_1(t)P_1 x+\psi (t,x) + \Pi (x)e(t),
\end{equation}
where $S_1\in C([0,\infty),L(X_1, Y_1))$, $\psi \in C([0,\infty)\times \Rn, Y_1)$, $\displaystyle\frac{\partial \psi }{\partial x} \in C([0,\infty)\times \Rn, L(\Rn ,Y_1))$, $e\in C([0,\infty),\Rn)$, $\Pi\in C^1 (\Rn ,L(\Rn ,Y_1))$ and there exist numbers $C,\, r >0$ such that $\|\Pi(x)\|\le C$ for all $\|P_1 x\|\ge r$.  Then Theorem~\textup{\ref{ThRF1}} remains valid if instead of \eqref{ineqreg} the following condition is satisfied:
\begin{equation}\label{consRF1-3}
(HP_1 x,G^{-1} \psi (t,x))\le 0\quad  \text{ for all }(t,x) \in L_0\, \text{ such that }\, 0\le t\le T,\; \| P_1 x\| \ge R_T.
\end{equation}
\end{corrolary}

\proof{Proof}
The proof of the corollary is similar to the proof of Theorem~1 from \cite{RF1}.
\endproof
\begin{definition}(\cite{Filnummeth}, \cite{Fil.MPhAG})\,
A solution $x(t)$ of the initial value problem \eqref{DAE}, \eqref{ini} is called \emph{Lagrange stable} if it is global and bounded, i.e., the solution  $x(t)$ exists on $[t_0,\infty )$ and $\mathop{\sup }\limits_{t\in [t_0,\infty )} \| x(t)\|<+\infty$.\;  \emph{The equation \eqref{DAE} is Lagrange stable} if every solution of the initial value problem \eqref{DAE}, \eqref{ini} is Lagrange stable.
\end{definition}

The theorem on the Lagrange stability of the semilinear DAE~\eqref{DAE} (\cite{Fil.MPhAG}) is given below.
\begin{theorem}\label{Th_Ust1}
\hspace{0pt}

\textup{I.} Let $f\in C([0,\infty)\times \Rn, \Rn)$,\, $\displaystyle\frac{\partial f}{\partial x}  \in C([0,\infty)\times \Rn, L(\Rn))$,\, $\lambda A+B$  be a regular pencil of index not higher than~1 and \eqref{soglreg2} be fulfilled.  Let for any $x^i_{p_2} \in X_2$ such that $(t_*, x_{p_1}^* + x^i_{p_2})\in L_0$, $i=1,2$, the operator function \eqref{funcPhi} be basis invertible on $[x^1_{p_2},\,x^2_{p_2}]$.  Assume that for some self-adjoint positive operator $H \in L(X_1)$ and some number $R>0$ there exist functions $k\in C([0,\infty),{\mathbb R})$, $U\in C((0,\infty), (0,\infty))$  such that $\displaystyle\int\limits_c^{+\infty} \frac{dv}{U(v)} =+\infty $ $(c>0)$ and
\begin{equation}\label{Lagr1}
\hspace*{-0.9cm} \big(HP_1 x,G^{-1}[-BP_1 x\! +\! Q_1 f(t,x)]\big)\!\le\! k(t)\, U\big({\textstyle \frac{1}{2}}(HP_1x,P_1x)\big) \text{ for all } (t,x)\!\in\! L_0 \text{ such that } \|P_1 x\|\!\ge\! R.
\end{equation}
Then for each initial point  $(t_0,x_0)\in L_0$ there exists a unique global solution  $x(t)$ of the initial value problem \eqref{DAE}, \eqref{ini}.

\textup{II.} If, additionally, $\displaystyle\int\limits_{t_0}^{+\infty} k(t)\, dt <+\infty$,
and there exists $\Tilde{x}_{p_2} \in X_2$ such that for any $x_{p_2}^*\in X_2$ satisfying $(t_*, x_{p_1}^* + x_{p_2}^*)\in L_0$  the operator function \eqref{funcPhi} is basis invertible on $(\Tilde x_{p_2},x_{p_2}^*]$ and the corresponding inverse operator is bounded uniformly in $x_{p_2}$, $t_*$,  and
\begin{equation}\label{LagrA1}
\mathop{\sup }\limits_{t\in [0,\infty ),\: \| x_{p_1}\|\le M} \|Q_2 f(t,x_{p_1}+ \tilde{x}_{p_2})\| < +\infty,\quad M>0\; (M\in {\mathbb R}),
\end{equation}
then for the initial points $(t_0,x_0)\in L_0$ the equation \eqref{DAE} is Lagrange stable.
\end{theorem}

\proof{Proof}
The proof is carried out in the same way as the proof of Theorem~3.1 from \cite{Fil.MPhAG}.
\endproof
\begin{remark}\label{RemDiff}
The solution $x(t)$ of the initial value problem \eqref{DAE}, \eqref{ini} is such that $P_1 x(t)\in C^1([t_0,\infty), X_1)$ and $P_2 x(t)\in  C([t_0,\infty),X_2)$.

If in Theorems~\ref{ThRF1},~\ref{Th_Ust1} $f\in C^m([0,\infty)\times \Rn, \Rn)$, $m\in {\mathbb N}$, then the solution $x(t)$  is such that $P_1 x(t)\in C^{m+1}([t_0,\infty), X_1)$ and $P_2 x(t)\in  C^m([t_0,\infty),X_2)$.
\end{remark}

The theorem on the Lagrange instability of the semilinear DAE~\eqref{DAE}, which gives conditions for the existence and uniqueness of solutions with a finite escape time (the solutions are blow-up in finite time) [\cite{Fil.MPhAG}, Definition~2.4] is also proved in [\cite{Fil.MPhAG}, Theorem~4.1]. Thus, this theorem gives the conditions under which the initial value problem \eqref{DAE}, \eqref{ini} does not have global solutions.

\section{The combined methods and their convergence}\label{NumMs}

We denote by $z=x_{p_1}$, $u=x_{p_2}$  the components of a vector $x=z+u\in \Rn$ \eqref{xr} ($z$, $u$ are the projections of the vector~$x$ onto subspaces $X_1$, $X_2$ from \eqref{Proj.2}).  We seek a solution of the initial value problem \eqref{DAE}, \eqref{ini} on an interval $[t_0,T]$. Introduce the uniform mesh $\omega_h=\{t_i=t_0+ih,\; i=0,...,N\}$ with the step size $h=(T-t_0)/N$.  The values of an approximate solution of the problem \eqref{DAE}, \eqref{ini} at the points $t_i$ are denoted by $x_i = z_i+u_i$, $i=0,...,N$ ($z_i=P_1 x_i$, $u_i=P_2 x_i$).  Initial values $z_0$, $u_0$ are chosen so that the consistency condition $B u_0 = Q_2 f(t_0,z_0+u_0)$, equivalent to the condition $(t_0,x_0)\in L_0$ of Theorems~\ref{ThRF1},~\ref{Th_Ust1}, is satisfied.

Recall that the projectors $P_1$, $P_2$, $Q_1$, $Q_2$ and the operator $G$ are computed by the formulas \eqref{ProjRes} and \eqref{Proj.G}.

\subsection{Method 1}\label{Num-meth}

\begin{theorem}\label{ThNum-meth}
Let the conditions of Theorem~\textup{\ref{ThRF1}} \textup{(}or Corollary~\textup{\ref{conseq-ThRF1}}\textup{)} or the conditions of part~\textup{I} of Theorem~\textup{\ref{Th_Ust1}} be satisfied.   Let, additionally, $f\in C^1([t_0,T]\times \Rn, \Rn)$, and the operator function $\Phi $ \eqref{funcPhi}\, \textup{(}where $t_*=t$, $x_{p_1}^*=z$, $x_{p_2}=u$\textup{)}  be invertible for any points $(t,z+u)\in [t_0,T]\times \Rn$.   Then the method
\begin{align}
 &x_0 = z_0 +u_0, \label{met1}\\
 &z_{i+1}=(I-h G^{-1} B) z_i+h G^{-1} Q_1 f(t_i,z_i+u_i), \label{met2}\\
 &u_{i+1}=u_i-\bigg[I - G^{-1}{\displaystyle \frac{\partial Q_2 f}{\partial x}} (t_{i+1},z_{i+1}+u_i)\bigg]^{-1}\! \big[u_i- G^{-1}Q_2f(t_{i+1},z_{i+1}+u_i)\big], \label{met3} \\
 &x_{i+1}=z_{i+1}+u_{i+1},\quad i=0,...,N-1, \label{met4}
\end{align}
approximating the initial value problem \eqref{DAE}, \eqref{ini} on $[t_0,T]$, converges and has the first order of accuracy \textup{(}$\mathop{\max }\limits_{0\le i\le N} \|z(t_i)-z_i\|=O(h)$, $\mathop{\max }\limits_{0\le i\le N} \|u(t_i)-u_i\|=O(h)$\textup{)}. \end{theorem}

\proof{Proof}
By virtue of the theorem conditions, there exists a unique global (exact) solution $x(t)$ of the initial value problem \eqref{DAE}, \eqref{ini} such that $z(t)=P_1 x(t)\in C^2([t_0,T],X_1)$ and $u(t)=P_2 x(t)\in C^1([t_0,T],X_2)$.

The equation \eqref{DAE} is equivalent to the system \eqref{prThreg1} which, with the new notations, is written in the form
\begin{align}
& \frac{dz}{dt} +G^{-1} Bz =G^{-1} Q_1 f(t,z+u), \label{num1}\\
& u =G^{-1} Q_2 f(t,z+u). \label{num2}
\end{align}
Using the Taylor formula, we obtain
\begin{align}
 &\frac{dz}{dt}(t)=\frac{z(t+h)-z(t)}{h}+O(h), \label{Taylor1}\\
 &\begin{array}{r}
 \displaystyle  G^{-1} Q_2 f(t+h,z(t+h)+u(t+h))=G^{-1} Q_2 f(t+h,z(t+h)+u(t)) + \smallskip\\
 \displaystyle  +G^{-1} \frac{\partial Q_2 f}{\partial x}(t+h,z(t+h)+u(t)) [u(t+h)-u(t)]+O(h). \end{array}\label{Taylor2}
\end{align}

From the theorem conditions and
\begin{equation*}
 -G^{-1}\bigg[\frac{\partial Q_2 f}{\partial x}(t,z+u)-B\bigg]P_2 = \bigg[I - G^{-1} \frac{\partial Q_2 f}{\partial x} (t,z+u)\bigg]P_2=\bigg[I - G^{-1} \frac{\partial Q_2 f}{\partial x} (t,z+u)\bigg]\bigg|_{X_2}
\end{equation*}
it follows that there exists the inverse operator $\bigg[I - G^{-1} \displaystyle\frac{\partial Q_2 f}{\partial x} (t,z+u)\bigg]^{-1}\in L(X_2)$ for any points ${(t,z+u)}$ from $L_0$ and $[t_0,T]\times \Rn$. Using \eqref{Taylor1}, \eqref{Taylor2} and taking into account the equivalence of the equation  \eqref{DAE} and the system \eqref{num1}, \eqref{num2}, we can write the problem \eqref{DAE}, \eqref{ini} at the points of the introduced mesh $\omega_h$ as
\begin{align}
 & x(t_0) = z(t_0)+u(t_0),\quad z(t_0) = z_0,\; u(t_0)=u_0, \label{exact1}\\
 & z(t_{i+1})=(I-h G^{-1} B) z(t_i)+h G^{-1} Q_1 f(t_i,z(t_i)+u(t_i))+O(h^2), \label{exact2}\\
 & \begin{array}{r}
 \displaystyle  u(t_{i+1}) = u(t_i)- \bigg[I- G^{-1} \frac{\partial Q_2 f}{\partial x}(t_{i+1},z(t_{i+1})+u(t_i))\bigg]^{-1}\! \big[u(t_i)- \smallskip\\
 \displaystyle  -G^{-1}Q_2 f(t_{i+1},z(t_{i+1})+u(t_i))\big]+ O(h), \end{array} \label{exact3} \\
 & x(t_{i+1})=z(t_{i+1})+u(t_{i+1}),\quad i=0,...,N-1, \label{exact4}
\end{align}
where \eqref{exact3} follows from
\begin{equation} \label{exact3*}
 \begin{array}{l}
 \displaystyle u(t_{i+1})= G^{-1}Q_2 f(t_{i+1},z(t_{i+1})+u(t_i))+{} \smallskip\\
 \displaystyle  {}+ G^{-1} \frac{\partial Q_2 f}{\partial x} (t_{i+1},z(t_{i+1})+u(t_i))[u(t_{i+1})-u(t_i)]+O(h). \end{array}
\end{equation}
The corresponding numerical method takes the form \eqref{met1}--\eqref{met4}  (\cite{Filnummeth}),  where \eqref{met3} is obtained from
\begin{equation}
u_{i+1}=G^{-1}Q_2 f(t_{i+1},z_{i+1}+u_i) +  G^{-1} \frac{\partial Q_2 f}{\partial x}(t_{i+1},z_{i+1} +  u_i)[u_{i+1}-u_i]. \label{met3*}
\end{equation}

Since $\frac{\partial f}{\partial x}(t,x)$ is continuous on $[0,\infty)\times \Rn$, then there exists the constant
\begin{equation}\label{M1}
M_1=\mathop{\sup }\limits_{0<\theta_1 <1} \bigg\|G^{-1}  \frac{\partial Q_1 f}{\partial x}(t_i,x_i + \theta_1(x(t_i)-x_i))\bigg\|
\end{equation}
such that (by the formula of finite increments) $\|G^{-1} Q_1[f(t_i,z(t_i)+u(t_i))- f(t_i,z_i+u_i)]\|\le M_1 \big(\|z(t_i)-z_i\|+ \|u(t_i)-u_i\|\big)$.   Then using \eqref{exact2}, \eqref{met2} we obtain the estimate
\begin{equation}\label{num3}
 \|z(t_{i+1})- z_{i+1}\|\le \big(\|I-h G^{-1} B\| + h M_1\big) \|z(t_i)- z_i\|+ h M_1\|u(t_i)-u_i\|+ O(h^2).
\end{equation}

Denote $\varepsilon ^z_{i+1} = \|z(t_{i+1})-z_{i+1}\|$, $\varepsilon ^u_{i+1}= \|u(t_{i+1})-u_{i+1}\|$, $g(h) = \|I-h G^{-1} B\| +h M_1$. Then  \eqref{num3} takes the form
\begin{equation}\label{num6}
 \varepsilon ^z_{i+1}\le g(h)\, \varepsilon ^z_i+h M_1\, \varepsilon ^u_i +O(h^2).
\end{equation}
It follows from the initial condition and \eqref{met2}, \eqref{exact2} that $\varepsilon ^z_0= 0$, $\varepsilon ^u_0 = 0$ and $\varepsilon ^z_1= O(h^2)$. Using \eqref{num6}, we find recurrently the estimate
\begin{equation}\label{num8}
\varepsilon ^z_{i+1}\le h M_1 \sum\limits_{j=0}^i g^{i-j}(h) \varepsilon ^u_j+ O(h^2)\sum\limits_{j=0}^i  g^j(h).
\end{equation}
Since $g^j(h) \le (1+h(\|G^{-1} B\|+M_1))^j \le e^{(T-t_0)(\|G^{-1} B\| + M_1)}$, $j=1,...,N$, then
\begin{equation}\label{num10}
 \varepsilon ^z_{i+1}\le O(h) \sum\limits_{j=0}^i \varepsilon ^u_j+ O(h),\quad i=0,...,N-1.
\end{equation}

Using \eqref{exact3*}, \eqref{met3*} we obtain
\begin{gather*}
u(t_{i+1}) - u_{i+1} =  \left[I-G^{-1} \frac{\partial Q_2 f}{\partial x}(t_{i+1},z_{i+1}+u_i) \right]^{-1}  \Big[ G^{-1}\frac{\partial Q_2f}{\partial x} (t_{i+1},z(t_{i+1}) + u(t_i)) [z(t_{i+1})-\\
-z_{i+1}] -  G^{-1}\frac{\partial Q_2f}{\partial x} (t_{i+1},z_{i+1} + u_i) [u(t_{i+1}) - u_i]+O(\|z(t_{i+1}) - z_{i+1}+u(t_i) - u_i \|)+O(h)\Big].
\end{gather*}
Denote $\displaystyle C_1=\! \mathop{\sup }\limits_{0\le i\le N-1}\! \left\|G^{-1}\frac{\partial Q_2f}{\partial x} (t_{i+1},z(t_{i+1})+u(t_i))\right\|\!$,  $\displaystyle C_2 =\!  \mathop{\sup }\limits_{0\le i\le N-1}\! \left\|G^{-1}\frac{\partial Q_2f}{\partial x} (t_{i+1},z_{i+1}+u_i)\right\|\!$,
\begin{equation}\label{C3}
C_3=\mathop{\sup }\limits_{0\le i\le N-1}\left\|\left[I-G^{-1}\frac{\partial Q_2f}{\partial x} (t_{i+1},z_{i+1}+u_i)\right]^{-1}\right\|.
\end{equation}
Then $\varepsilon ^u_{i+1}\le C_3 [ C_1 \varepsilon ^z_{i+1}+C_2\varepsilon ^u_i+ (C_2+1)O(h)+O(\varepsilon ^z_{i+1}+\varepsilon ^u_i)]$. Consequently, there exist constants $\alpha $, $\beta $ such that
\begin{equation}\label{num5}
\varepsilon ^u_{i+1}\le \alpha \varepsilon ^z_{i+1}+ \beta \varepsilon ^u_i +O(h),\quad i=0,...,N-1.
\end{equation}
From \eqref{num5} we obtain that $\varepsilon ^u_1\le O(h)$ and, taking into account \eqref{num10},  $\varepsilon ^u_{i+1}\le O(h) \sum\limits_{j=1}^i \varepsilon ^u_j+ \beta \varepsilon ^u_i +O(h)$, $i=1,...,N-1$.  By using the method of mathematical induction, it is easy to prove that
$\varepsilon ^u_{i+1}\le O(h)$, $i=0,...,N-1$.  Then from \eqref{num10} it follows that $\varepsilon ^z_{i+1}\le O(h)$, $i=0,...,N-1$. Thus, $\mathop{\max }\limits_{0\le i\le N} \varepsilon ^u_i = O(h)$ and $\mathop{\max }\limits_{0\le i\le N} \varepsilon ^z_i = O(h)$. Hence, $\mathop{\max }\limits_{0\le i\le N}  \left\| \left( \begin{array}{c} z(t_i)-z_i \\ u(t_i)-u_i \end{array}  \right) \right\|  \le K h$, where $K >0$ is some constant, and the method \eqref{met1}--\eqref{met4} converges and has the first order of accuracy.
\endproof
\begin{remark}\label{remNum-meth}
If in Theorem~\ref{ThNum-meth} we do not require the additional smoothness of the function $f$, i.e., we assume that $f\in C([0,\infty)\times \Rn, \Rn)$ and $\displaystyle\frac{\partial f}{\partial x}  \in C([0,\infty)\times \Rn, L(\Rn))$, then the method \eqref{met1}--\eqref{met4} converges, but may not have the first order of accuracy:  $\mathop{\max }\limits_{0\le i\le N} \|z(t_i)-z_i\|=o(1)$, $h\to 0$\, (i.e., $\mathop{\max }\limits_{0\le i\le N} \|z(t_i)-z_i\|\to 0$, $h\to 0$),\; $\mathop{\max}\limits_{0\le i\le N} \|u(t_i)-u_i\|=o(1)$, $h\to 0$.
\end{remark}
\proof{The proof of Remark~\ref{remNum-meth}}  The proof is carried out in the same way as the proof of Theorem~\ref{ThNum-meth}, where instead of \eqref{Taylor1}, \eqref{Taylor2} we use the representations
\begin{equation}\label{Taylor1o1}
\frac{dz}{dt}(t)=\frac{z(t+h)-z(t)}{h}+o(1),\; h\to 0,
\end{equation}
\begin{equation}\label{Taylor2o1}
 \begin{array}{c}
 \displaystyle G^{-1} Q_2 f(t+h,z(t + h)+u(t + h)) = G^{-1} Q_2 f(t + h,z(t+h)+u(t))+  \smallskip\\
 \displaystyle +G^{-1} \frac{\partial Q_2 f}{\partial x}(t+h,z(t+h)+u(t)) [u(t+h)-u(t)]+o(1),\; h\to 0. \end{array}
\end{equation}
\endproof

\subsection{Method 2 (modified method 1)}\label{ModNum-meth}
\begin{theorem}\label{ThModNum-meth}
Let the conditions of Theorem~\textup{\ref{ThRF1}} \textup{(}or Corollary~\textup{\ref{conseq-ThRF1}}\textup{)} or the conditions of part~\textup{I} of Theorem~\textup{\ref{Th_Ust1}} be satisfied.  Let, additionally, $f\in C^2([t_0,T]\times \Rn, \Rn)$, and the operator function $\Phi $ \eqref{funcPhi}\, \textup{(}where $t_*=t$, $x_{p_1}^*=z$, $x_{p_2}=u$\textup{)} be invertible for any points $(t,z+u)\in [t_0,T]\times \Rn$.  Then the method
\begin{align}
\hspace*{-0.3cm} &x_0 = z_0 +u_0,\quad z_1=(I-h G^{-1} B) z_0+h G^{-1} Q_1 f(t_0,z_0+u_0), \label{Impmet1}\\
\hspace*{-0.3cm} &u_{i+1}\!=\! u_i-\!\bigg[I - G^{-1}{\displaystyle \frac{\partial Q_2 f}{\partial x}} (t_{i+1},z_{i+1}+u_i)\bigg]^{-1}\! \big[u_i- G^{-1}Q_2f(t_{i+1},z_{i+1}+u_i)\big],\, i=0,...,N-1, \label{Impmet3}\\
\hspace*{-0.3cm} &z_{i+1}=z_{i-1} + 2h G^{-1}[Q_1 f(t_i,z_i+u_i) -B z_i],\: i=1,...,N-1, \label{Impmet2}\\
\hspace*{-0.3cm} &x_{i+1}=z_{i+1}+u_{i+1},\; i=0,...,N-1, \label{Impmet4}
\end{align}
approximating the initial value problem \eqref{DAE}, \eqref{ini} on $[t_0,T]$, converges and has the second order of accuracy\,   \textup{(}$\mathop{\max }\limits_{0\le i\le N} \|z(t_i)-z_i\|=O(h^2)$, $\mathop{\max }\limits_{0\le i\le N} \|u(t_i)-u_i\|=O(h^2)$\textup{)}.
\end{theorem}
\proof{Proof}
By virtue of the theorem conditions, there exists a unique global (exact) solution $x(t)$ of the initial value problem \eqref{DAE}, \eqref{ini} such that $z(t)=P_1 x(t)\in C^3([t_0,T],X_1)$ and $u(t)=P_2 x(t)\in C^2([t_0,T],X_2)$.

Taking into account the smoothness of the solution and the function $f$, we obtain
\begin{align}
&\frac{dz}{dt}(t)=\frac{z(t+h)-z(t-h)}{2h}+O(h^2), \label{Taylor3}\\
&\begin{array}{r}
 \displaystyle G^{-1} Q_2 f(t+h,z(t+h)+u(t+h))=G^{-1} Q_2 f(t+h,z(t+h)+u(t)) + \smallskip\\
 \displaystyle +G^{-1} \frac{\partial Q_2 f}{\partial x}(t+h,z(t+h)+u(t)) [u(t+h)-u(t)]+O(h^2).\end{array}\label{Taylor4}
\end{align}

As shown in the proof of Theorem~\ref{ThNum-meth}, there exists the inverse operator $\bigg[I - G^{-1} \displaystyle\frac{\partial Q_2 f}{\partial x} (t,z+u)\bigg]^{-1}\in L(X_2)$ for any points  $(t,z+u)$  from $L_0$ and $[t_0,T]\times \Rn$.   Using \eqref{Taylor1}, \eqref{Taylor3}, \eqref{Taylor4} and the system \eqref{num1}, \eqref{num2} equivalent to the equation~\eqref{DAE}, we can write the problem \eqref{DAE}, \eqref{ini} at the points of the introduced mesh $\omega_h$ as
\begin{align}
 & \begin{array}{l}
  \displaystyle x(t_0) = z(t_0)+u(t_0),\; z(t_0) = z_0,\; u(t_0)=u_0,  \smallskip\\
  \displaystyle  z(t_1)=(I-h G^{-1} B) z_0+h G^{-1} Q_1 f(t_0,z_0+u_0)+O(h^2).\end{array}  \label{Impexact1}\\
 & \begin{array}{l}
  \displaystyle  u(t_{i+1}) = u(t_i)- \bigg[I- G^{-1} \frac{\partial Q_2 f}{\partial x}(t_{i+1},z(t_{i+1})+u(t_i))\bigg]^{-1}\! \big[u(t_i)-  \smallskip\\
  \displaystyle  -G^{-1}Q_2 f(t_{i+1},z(t_{i+1})+u(t_i))\big]+  O(h^2),\; i=0,...,N-1,  \end{array}  \label{Impexact3} \\
 & z(t_{i+1})=z(t_{i-1}) + 2h G^{-1}[Q_1 f(t_i,z(t_i)+u(t_i))  - B z(t_i)] + O(h^3),\; i = 1,...,N-1, \label{Impexact2}\\
 & x(t_{i+1})=z(t_{i+1})+u(t_{i+1}),\; i = 0,...,N-1, \label{Impexact4}
\end{align}
The corresponding numerical method takes the form   \eqref{Impmet1}--\eqref{Impmet4},  where $z_1$ coincides with $z_1$ from \eqref{met2} and \eqref{Impmet3} is obtained from \eqref{met3*}.

Denote $\varepsilon ^z_{i+1}= \|z(t_{i+1})-z_{i+1}\|$, $\varepsilon ^u_{i+1}= \|u(t_{i+1})-u_{i+1}\|$. Obviously, $\varepsilon ^z_0= 0$, $\varepsilon ^u_0 = 0$, $\varepsilon ^z_1= O(h^2)$, and there exists the constant $M_1$ \eqref{M1} such that
\begin{equation}\label{epsz1}
 \varepsilon ^z_{i+1}\le \varepsilon ^z_{i-1}+2h(\|G^{-1} B\| +M_1)\varepsilon ^z_i+ 2 h M_1\, \varepsilon ^u_i +O(h^3).
\end{equation}

Consider the system consisting of the inequality \eqref{epsz1} and equality $\varepsilon ^z_i=\varepsilon ^z_i$. Denoting
$\Hat\varepsilon ^z_{i+1}=\begin{pmatrix} \varepsilon ^z_{i+1} \\ \varepsilon ^z_i\end{pmatrix}$, $e_1 =\begin{pmatrix} 1 \\ 0\end{pmatrix}$, $\Hat g(h)=\begin{pmatrix} 2h(\|G^{-1} B\|+M_1) & 1 \\ 1 & 0 \end{pmatrix}$,  we obtain the representation of this system as
\begin{equation}\label{hatepsz}
 \Hat\varepsilon ^z_{i+1}\le  \Hat g(h) \Hat\varepsilon ^z_i+ 2h M_1 e_1 \varepsilon ^u_i +e_1 O(h^3),\quad i=1,...,N-1.
\end{equation}
Using \eqref{hatepsz}, we find recurrently the estimate
\begin{equation}\label{hatepsz2}
 \Hat\varepsilon ^z_{i+1}\le  \Hat g^i(h) \Hat\varepsilon ^z_1 +2 h M_1 \sum\limits_{j=1}^i \Hat g^{i-j}(h) e_1\varepsilon ^u_j+ O(h^3)\sum\limits_{j=1}^i \Hat g^{j-1}(h) e_1.
\end{equation}
Define the norm of $\Hat\varepsilon ^z_{i+1}$ as $\|\Hat\varepsilon ^z_{i+1}\|=\max\{|\varepsilon ^z_{i+1}|, |\varepsilon ^z_i|\}$.  Then $\|e_1\|=1$, $\|\Hat\varepsilon ^z_1\|=\varepsilon ^z_1=O(h^2)$,  ${\|\Hat g(h)\|=1+2h(\|G^{-1} B\| +M_1)}$,\,  $\varepsilon ^z_{i+1}\le \|\Hat\varepsilon ^z_{i+1}\|$ and consequently
\begin{equation}\label{num8new}
 \varepsilon ^z_{i+1}\le  2 h M_1 \sum\limits_{j=1}^i \|\Hat g(h)\|^{i-j} \varepsilon ^u_j+ O(h^3)\sum\limits_{j=1}^i \|\Hat g(h)\|^{j-1}+ O(h^2).
\end{equation}
Since $\|\Hat g(h)\|^j \le e^{2(T-t_0)(\|G^{-1} B\| + M_1)}$, $j=1,...,N-1$, then
\begin{equation}\label{epsz2}
 \varepsilon ^z_{i+1}\le O(h) \sum\limits_{j=1}^i \varepsilon ^u_j+ O(h^2),\quad i=1,...,N-1.
\end{equation}

Using \eqref{Taylor4}, \eqref{met3*}  we obtain
\begin{gather*}
u(t_{i+1})-u_{i+1}=\left[I-G^{-1} \frac{\partial Q_2 f}{\partial x}(t_{i+1},z_{i+1}+u_i) \right]^{-1}\Big[z(t_{i+1})-z_{i+1} + \\
+O(\|z(t_{i+1})-z_{i+1}+u(t_i)-u_i \|^2+\|z(t_{i+1})-z_{i+1}+u(t_i)-u_i \|\, \|u(t_{i+1})-u(t_i) \|)+O(h^2)\Big].
\end{gather*}
Denote $C_3$ as \eqref{C3}. It is easy to verify that  $\varepsilon ^u_{i+1}\le 3C_3 \big[\varepsilon ^z_{i+1}+O\big((\varepsilon ^z_{i+1})^2+ (\varepsilon ^u_i)^2\big)+O(h^2)\big]$.  Taking into account \eqref{epsz2}, we easy obtain  $\varepsilon ^u_{i+1}\le O(h) \sum\limits_{l=1}^i \varepsilon ^u_l +O(h^2)\Big(\sum\limits_{l=1}^i \varepsilon ^u_l\Big)^2+\Hat C_3 (\varepsilon ^u_i)^2 + O(h^2)$,  where $\Hat C_3 $ is some constant. From the obtained estimate, using the polynomial (multinomial) theorem, we obtain
\begin{equation}\label{epsu1}
\varepsilon ^u_{i+1}\le O(h) \sum\limits_{l=1}^i \big(\varepsilon ^u_l+ (\varepsilon ^u_l)^2 \big) +\Hat C_3 (\varepsilon ^u_i)^2 + O(h^2),\quad i=1,...,N-1.
\end{equation}
Since $\varepsilon ^u_1\le O(h^2)$, then, by using the method of mathematical induction, it is easy to prove that $\varepsilon ^u_{i+1}\le O(h^2)$, $i=1,...,N-1$. From the obtained estimate and \eqref{epsz2} it follows that $\varepsilon ^z_{i+1}\le O(h^2)$, $i=1,...,N-1$.  Thus, $\mathop{\max }\limits_{0\le i\le N} \varepsilon ^u_i= O(h^2)$ and $\mathop{\max }\limits_{0\le i\le N} \varepsilon ^z_i= O(h^2)$.  Hence, the method \eqref{Impmet1}--\eqref{Impmet4} converges and has the second order of accuracy.
\endproof
\begin{remark}\label{remModNum-meth}
Analogously to Remark~\ref{remNum-meth}, if in Theorem~\ref{ThModNum-meth} $f\in C([0,\infty)\times \Rn, \Rn)$ and $\displaystyle\frac{\partial f}{\partial x}  \in C([0,\infty)\times \Rn, L(\Rn))$, then the method \eqref{Impmet1}--\eqref{Impmet4} converges:
${\mathop{\max }\limits_{0\le i\le N} \|z(t_i)-z_i\|=o(1)}$,~${h\to 0}$, ${\mathop{\max }\limits_{0\le i\le N} \|u(t_i)-u_i\|=o(1)}$,~${h\to 0}$.
\end{remark}
\proof{The proof of Remark~\ref{remModNum-meth}}
The proof is carried out in the same way as the proof of Theorem~\ref{ThModNum-meth}, where instead of \eqref{Taylor1}, \eqref{Taylor3}, \eqref{Taylor4} we use the representations \eqref{Taylor1o1}, $\displaystyle\frac{dz}{dt}(t)=\frac{z(t+h)-z(t-h)}{2h}+o(1)$,~$h\to 0$, and~\eqref{Taylor2o1}.
\endproof
\begin{remark}\label{globBasInv}
If the operator function $\Phi$ \eqref{funcPhi} is basis invertible on $[x^1_{p_2},\,x^2_{p_2}]$ for any $x^1_{p_2},\, x^2_{p_2}\in  X_2$, $t_*\in [0,\infty)$, $x^*_{p_1}\in X_1$, then the additional condition of the invertibility of $\Phi$ for any $(t,z+u)\in [t_0,T]\times \Rn$ from Theorems~\ref{ThNum-meth},~\ref{ThModNum-meth}, as well as the additional condition of the basis invertibility of $\Phi$ on $(\Tilde x_{p_2},x_{p_2}^*]$ from part~\textup{II} of Theorem~\textup{\ref{Th_Ust1}}, are not needed.
\end{remark}

\section{The application of the numerical methods to the analysis of mathematical models}\label{Appl}

Differential-algebraic (descriptor, degenerate or implicit differential) equations have a wide range of practical applications.  Such equations arise from the mathematical modeling of the dynamics of physical, economic, technical, ecological and other processes because of the availability of algebraic connections between the coordinates of the vectors of states of the corresponding dynamical systems.  Semilinear DAEs are used in modeling the dynamics of complex mechanical and robotic systems (\cite{RabierRh}, \cite{FoxJenZom}), various descriptor systems and neural networks (\cite{RiazaZufiria}, \cite{KunkelMehrmann}, \cite{EngwerdaSalmW}), transient processes in electrical circuits (\cite{Riaza}, \cite{BrenanCP}, \cite{LamourMarzTisch}, \cite{RF2}), an inter-industry balance, and other objects and processes (\cite{Vlasenko1}).

Consider the mathematical model designed in \cite{Fil.MPhAG} for a certain nonlinear electrical circuit.  This model is described by the semilinear DAE \eqref{DAE}, where $x=(x_1, x_2, x_3)^T=(I_L, U_C, I)^T$,
\begin{equation}\label{two-pole}
 A=\begin{pmatrix} L & 0 & 0 \\ 0 & C & 0 \\ 0 & 0 & 0 \end{pmatrix},\,
 B=\begin{pmatrix} 0 & 1 & r \\ 0 & g & -1 \\ 0 & 1 & r \end{pmatrix},\,
 f(t,x)=\begin{pmatrix} e(t)-\varphi_0(x_1)-\varphi(x_3) \\ -h(x_2) \\ \psi(x_1-x_3)-\varphi(x_3) \end{pmatrix}.
\end{equation}
Here $I_L$, $I$  are unknown currents, and $U_C$ is an unknown voltage. An input voltage $e$ is given.  The remaining currents and voltages in the circuit are uniquely expressed in terms of  $I_L$, $U_C$ and $I$.  Also, here $L$ is an inductance, $C$  is a capacitance, $r$  is a linear resistance, $g$  is a linear conductance, $\varphi$, $\varphi _0$, $\psi$ are nonlinear resistances, and $h$ is a nonlinear conductance.  It is assumed that $L,\, C,\, r,\, g>0$, $x\in {\mathbb R}^3$, $e \in C([0,\infty ),{\mathbb R})$, and $\varphi _0 $, $\varphi $, $\psi$, $h \in C^1 ({\mathbb R})$.  Inductance is given in henries (H), capacitance is given in farads (F), resistance is given in ohm ($\Omega$), and conductance is given in siemens (S).  When computations are carried out, it is assumed that inductance, capacitance and time are given in $\mu$H (microhenries), $\mu$F and $\mu$s, respectively.    For example, in Section~\ref{CompareMeth}, the parameters $L = 500$, $C = 0.5$  ($L = 500\; \mu\text{H} =5\cdot 10^{-4} \text{ H}$, $C =0.5\; \mu\text{F}=5\cdot 10^{-7} \text{ F}$) are used to carry out computations and it is assumed that the numerical solutions are obtained for the initial parameters on time intervals given in $\mu$s ($1\, \mu\text{s}=10^{-6} \text{ s}$).  It is easy to verify that this transition is true. It is done in order to the small values of $L$ and $C$ do not lead to large roundoff errors in the computations.

The conditions for the unique global solvability and the Lagrange stability of the DAE \eqref{DAE} with \eqref{two-pole} for arbitrary $\varphi _0$, $\varphi$, $\psi$, $h$, $e$ are given in \cite{Fil.MPhAG}. Below we present some classes of functions for which these conditions are satisfied.

In \cite{Fil.MPhAG} it is proved that for each initial point $(t_0,x^0)$, satisfying the consistency condition $x_2^0 +r x_3^0= \psi(x_1^0-x_3^0)-\varphi(x_3^0)$, there exists a unique global solution of the DAE~\eqref{DAE} with \eqref{two-pole} and the initial condition $x(t_0)=x^0$  for the functions of the form $\varphi_0(x_1)=\alpha _1 x_1^{2k-1}$, $\varphi(x_3)=\alpha _2 x_3^{2l-1}$, $\psi(x_1-x_3)=\alpha _3 (x_1-x_3)^{2j-1}$, $h(x_2)=\alpha _4 x_2^{2s-1}$,\: $k,l,j,s\in \mathbb N$, $\alpha _i >0$, $i=1,...,4$,\; if $j\le k$, $j\le s$ and $\alpha _3$ is sufficiently small, and for the functions of the form $\varphi _0 (x_1)=\alpha _1 x_1^{2k-1}$, $\varphi(x_3)=\alpha _2 \sin x_3$, $\psi(x_1-x_3)=\alpha _3\sin (x_1-x_3)$, $h(x_2)=\alpha _4 \sin x_2$ (instead of $\sin$ one can use $\cos$),\; if ${\alpha _2 + \alpha _3 < r}$.   If, additionally, $\mathop{\sup }\limits_{t\in [0,\infty )}|e(t)| < +\infty$ or $\int\limits_{t_0}^{+\infty}  |e(t)|\, dt <+\infty$,  then for the initial points $(t_0,x^0)$ the DAE is Lagrange stable (in both cases), i.e., every its solution is bounded. These conditions are fulfilled, for example, for voltages of the form $e(t)=\beta\sin({\omega t+\theta})$, $e(t)=\beta(t+\alpha)^{-n}$, $e(t)=\beta e^{-\alpha t}$, $\alpha> 0$, $\beta,\, \omega \in {\mathbb R}$, $n \in {\mathbb N}$, $\theta \in [0, 2\pi]$. For the voltage $e(t)= \beta(t+\alpha)^n$,\, $\alpha,\,\beta \in {\mathbb R}$, $n \in {\mathbb N}$, global solutions exist, but they are not bounded on the whole domain of definition.

The plots of numerical solutions for some particular cases are presented below. The parameters are given in the units mentioned above and time $t$ is given in~$\mu$s.

Choose the linear parameters $L = 5\cdot 10^{-4}$, $C = 5\cdot 10^{-7}$,  $r = 2$, $g =0.2$, the nonlinear resistances $\varphi_0 (x_1)= x_1^3$, $\varphi(x_3)= \sin x_3$, $\psi(x_1-x_3)=\sin (x_1-x_3)$ and conductance $h(x_2)= \sin x_2$, the input voltage $e(t)=(2 t+10)^{-2}$ and the initial data $t_0 = 0$, $x_0 =(10,-10,5)^T$. In this case, the solution is Lagrange stable, i.e., global and bounded.  The plots for the numerical solution are presented in Figure~\ref{MM_Ex5_10-12}.
\begin{figure}[H]%
\centering
\includegraphics[width=6cm]{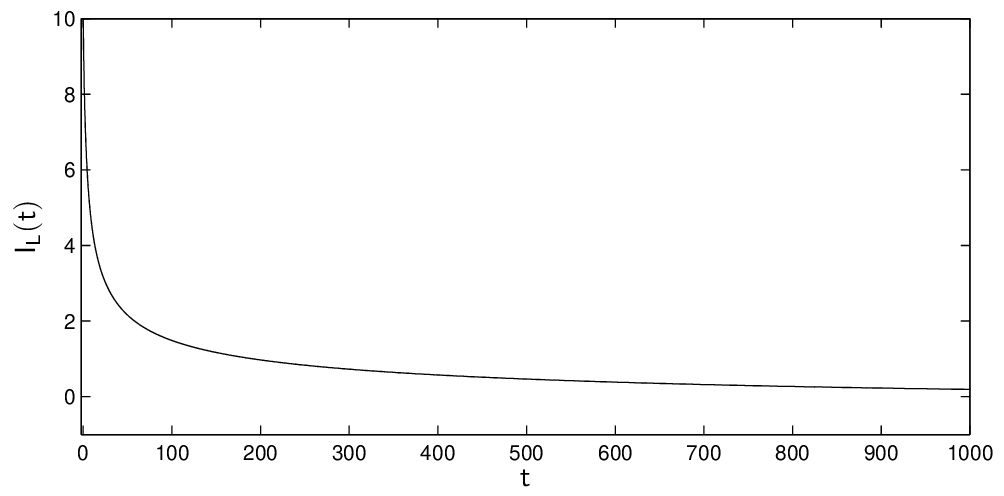}
\includegraphics[width=6cm]{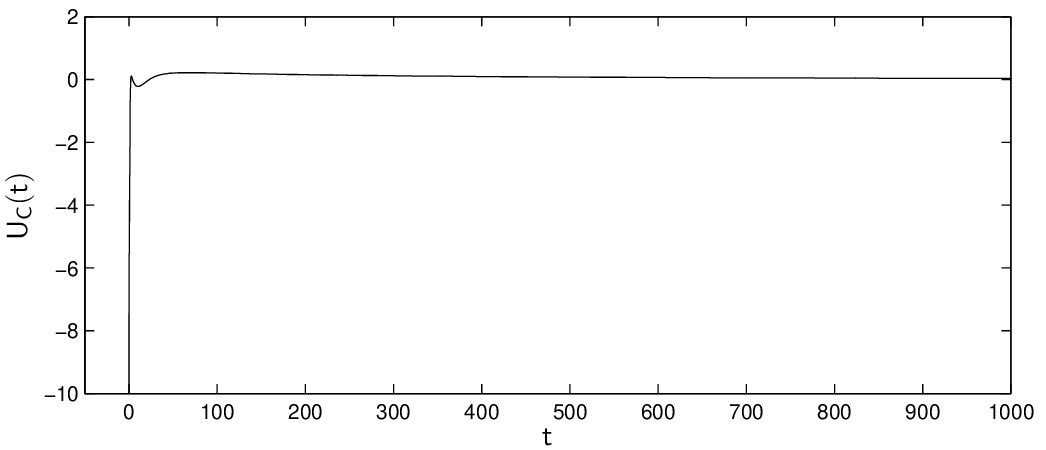}
\includegraphics[width=6cm]{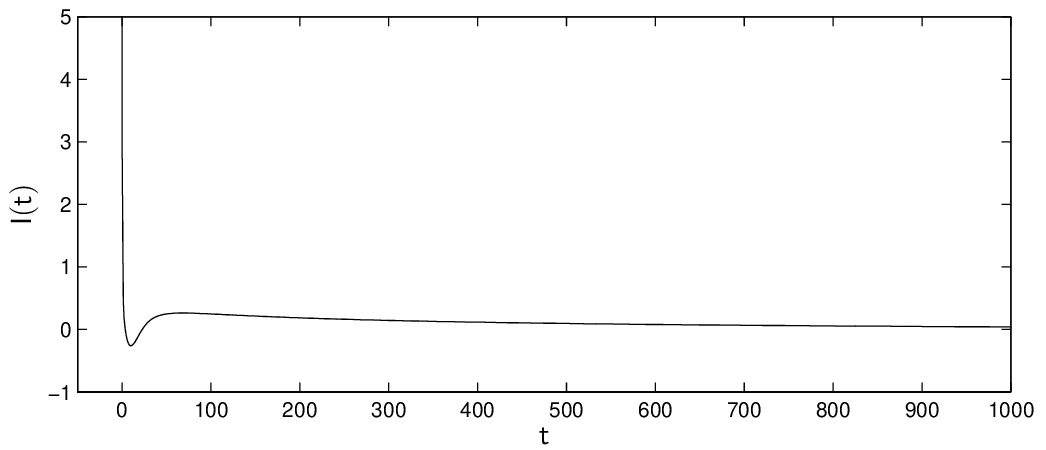}
\caption{The components of the numerical solution}\label{MM_Ex5_10-12}
\end{figure}

If $e(t) = t^2$ and $L= 10^{-3}$, $C = 5\cdot 10^{-7}$, $r =2$, $g =0.3$, $\varphi_0(x_1) = x_1^3$, $\varphi(x_3)= x_3^3$, $h(x_2)= x_2^3$, $\psi(x_1-x_3)= (x_1-x_3)^3$, $t_0 = 0$, $x_0 =(0,0,0)^T$, then a solution is global, but not bounded. The components of the numerical solution are shown in Figure~\ref{MM_Ex5_1-3k}.
\begin{figure}[H]%
\centering
\includegraphics[width=6cm]{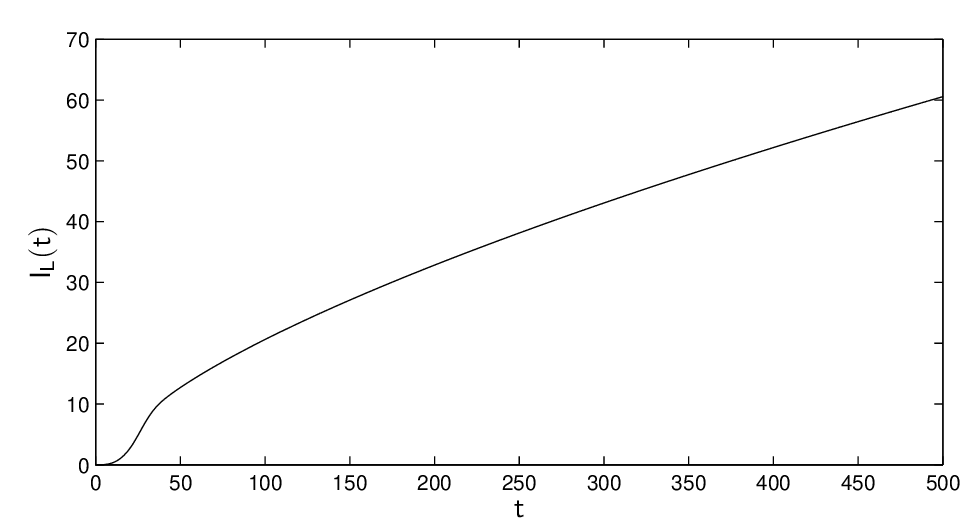}
\includegraphics[width=6cm]{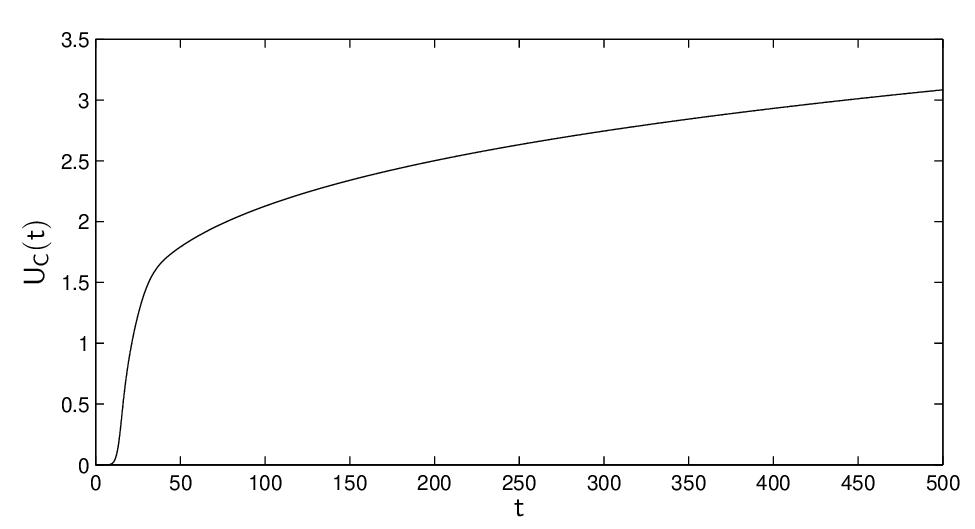}
\includegraphics[width=6cm]{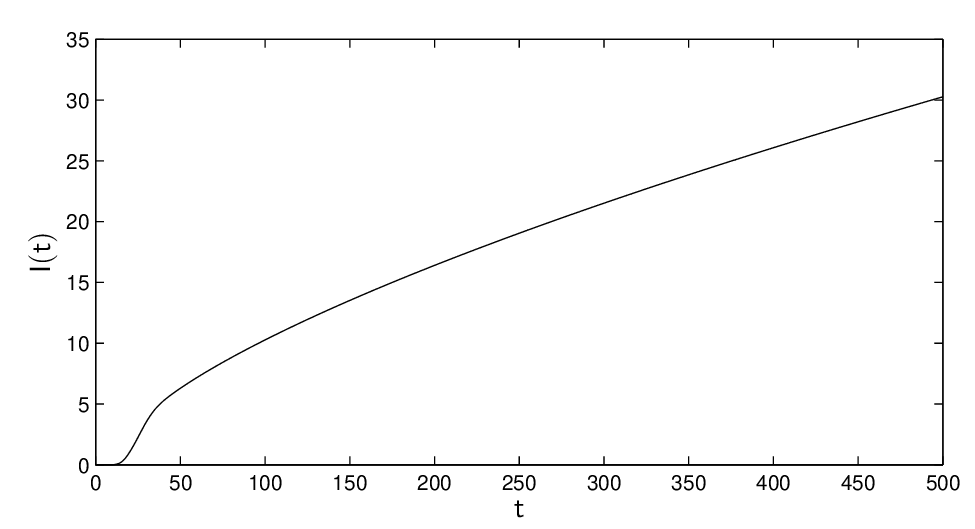}
\caption{The components of the numerical solution}\label{MM_Ex5_1-3k} \end{figure}

For the voltage of the triangular shape  $e(t)= 50-|t-50-100 k|$, $t\in [100\, k, 100+ 100\, k]$, $k\in {\{0\}\cup {\mathbb N}}$, and  $L=5\cdot 10^{-4}$, $C=5\cdot 10^{-7}$, $r=2$, $g=0.2$, $\varphi_0(x_1) = x_1^3$, $\varphi(x_3)= x_3^3$, $h(x_2)= x_2^3$, $\psi(x_1-x_3)= (x_1-x_3)^3$, $t_0 = 0$, $x_0 =(0,0,0)^T$, the solution is Lagrange stable. The solution components are presented in Figure~\ref{mTriang3h}.
\begin{figure}[H]%
\centering
\includegraphics[width=6cm]{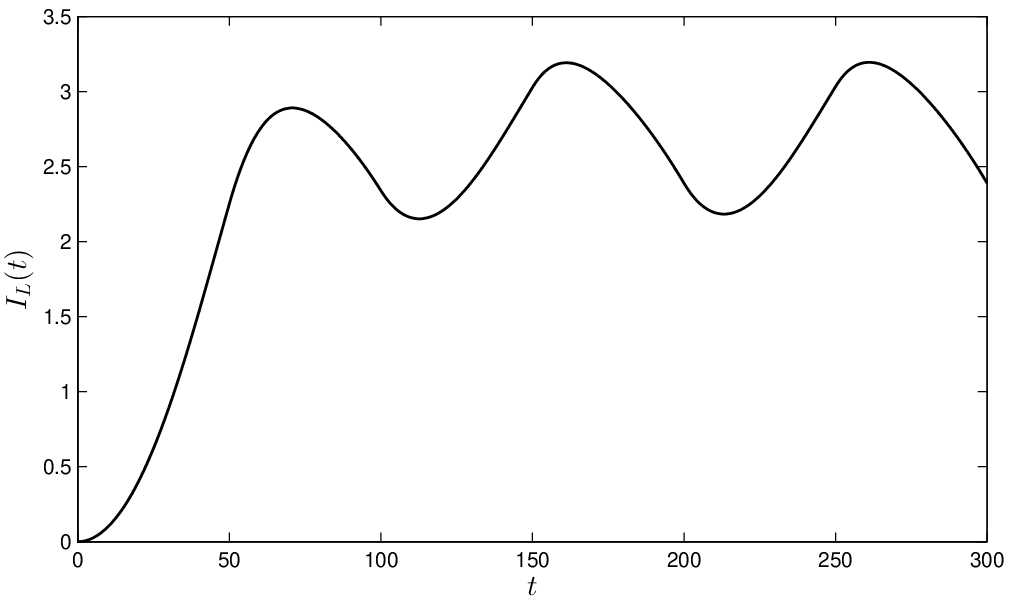}
\includegraphics[width=6cm]{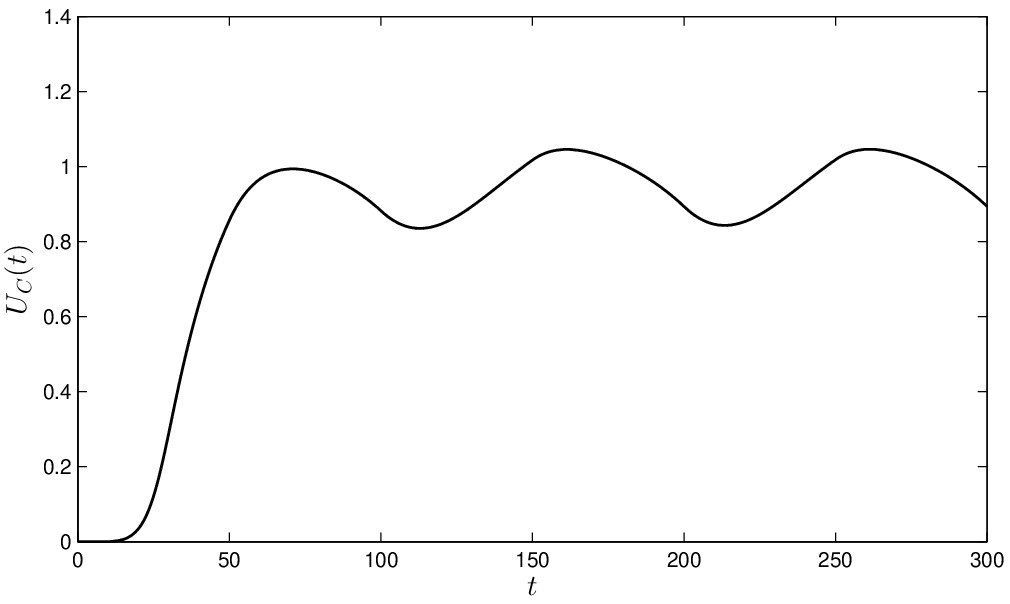}
\includegraphics[width=6cm]{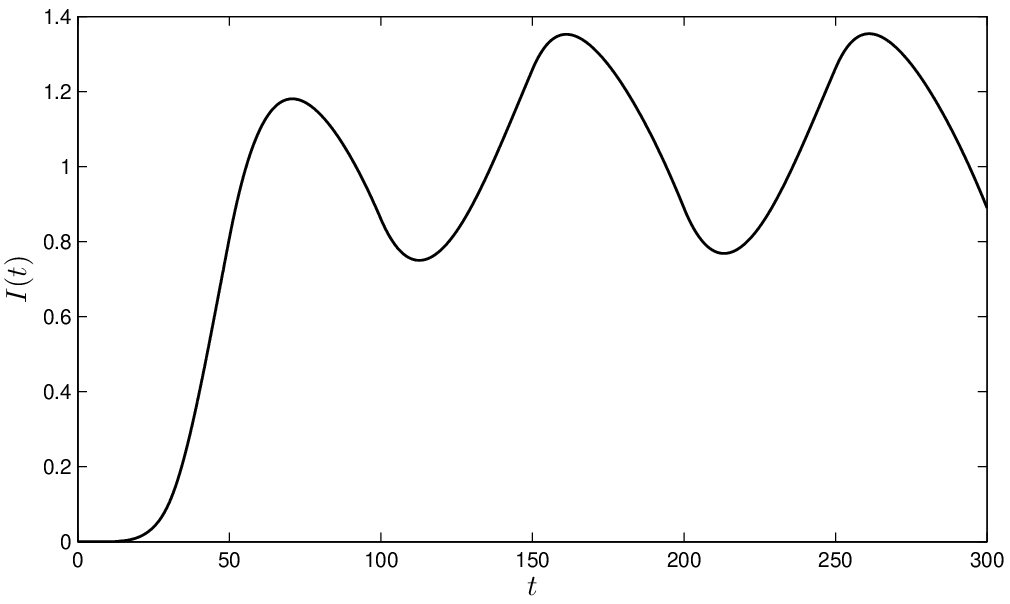}
\caption{The components of the numerical solution}\label{mTriang3h}
\end{figure}

Below is the example of a Lagrange-unstable solution (Figure~\ref{MM_Ex6_109-111}), which is blow-up in finite time (the norm of the solution tends to infinity on a finite time interval). Here $L=5\cdot 10^{-6}$, $C = 5\cdot 10^{-7}$, $r = 2$, $g =0.2$, $\varphi_0(x_1) =-x_1^2$, $\varphi(x_3)=x_3^3$, $h(x_2)=x_2^2$, $\psi(x_1-x_3)=(x_1-x_3)^3$, $e(t) = 2\, \sin t$, and $t_0 = 0$, $x_0 =(1, -6.5, 1.5)^T$.
\begin{figure}[H]%
\centering
\includegraphics[width=4.8cm]{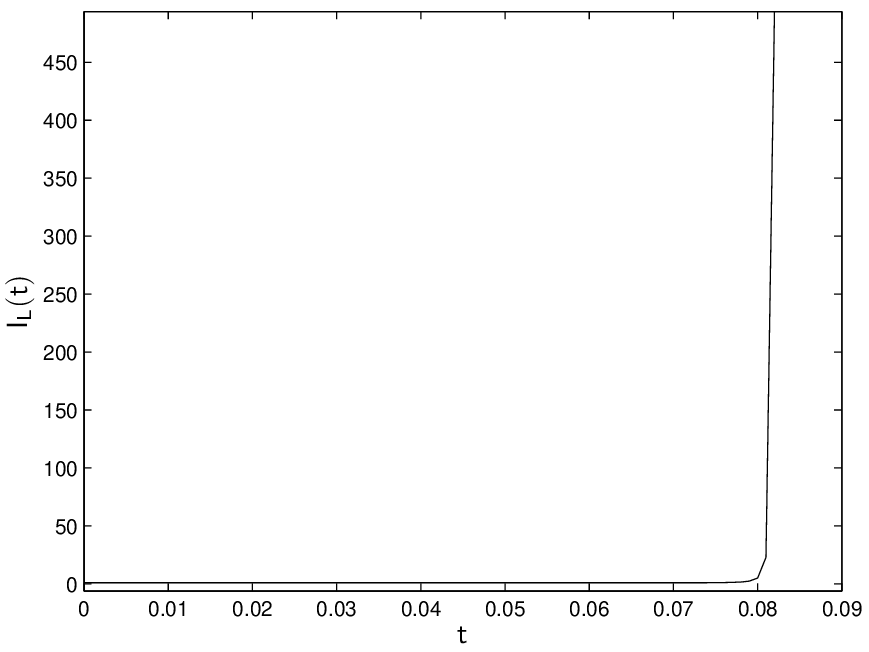}
\includegraphics[width=4.8cm]{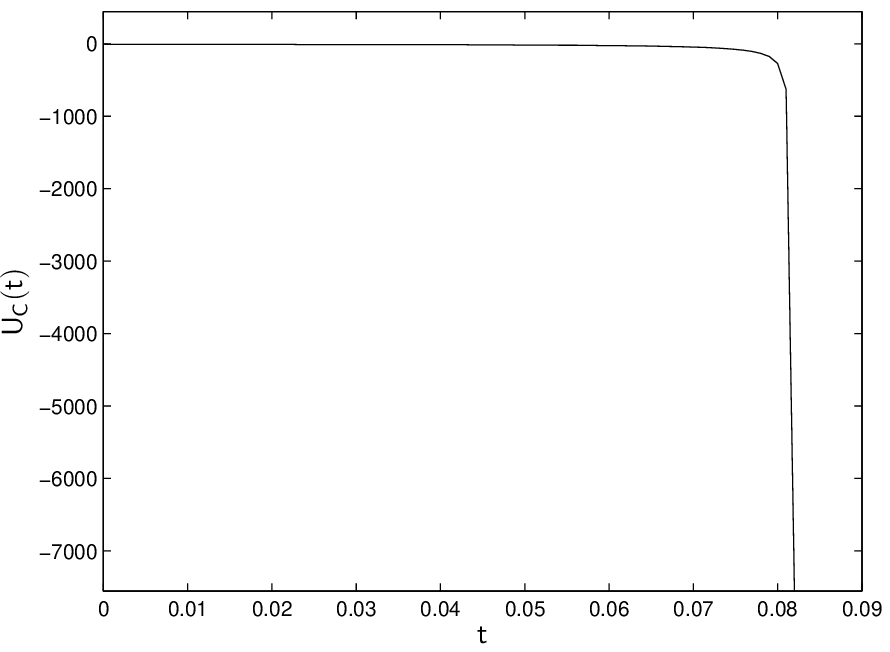}
\includegraphics[width=4.8cm]{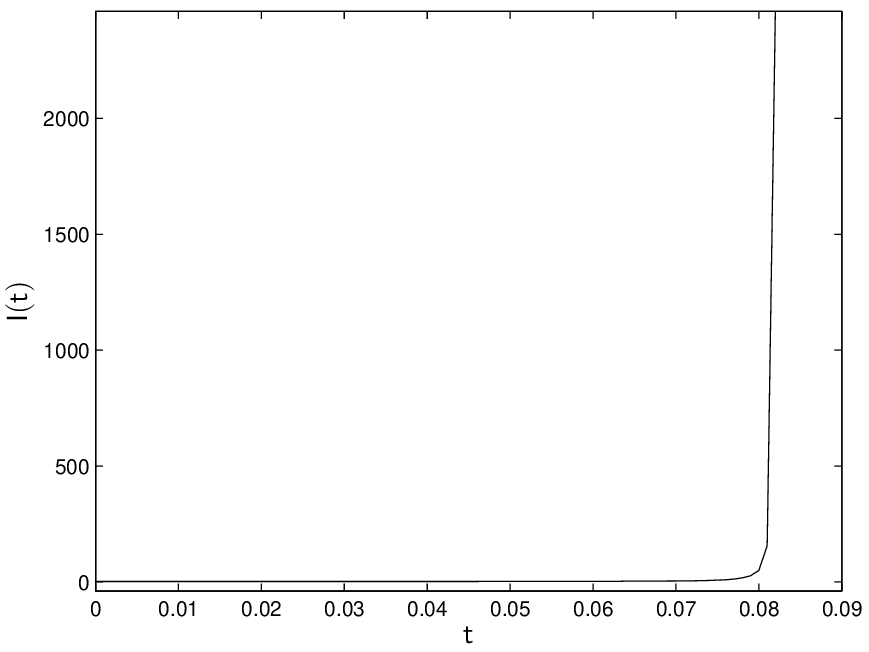}
\caption{The components of the numerical solution}\label{MM_Ex6_109-111}
\end{figure}

The analysis of the obtained numerical solutions verifies the results of theoretical studies.

Thus, the developed numerical methods allow one to carry out the numerical analysis of global dynamics for the mathematical models described by semilinear DAEs and to draw the conclusions about their evolutionary properties.   Also, the developed methods have a number of advantages, which have already been mentioned in Section~\ref{Intro}.

\section{Comparison of methods 1, 2}\label{CompareMeth}

Consider the semilinear DAE \eqref{DAE}, where the matrices $A$, $B$ and nonlinear function $f$ have the form \eqref{two-pole}, and $x=(x_1,x_2,x_3)^T=(I_L, U_C, I)^T$.    Let $L = 5\cdot 10^{-4}$, $C = 5\cdot 10^{-7}$, $r = 2$, $g =0.2$, $e(t) = \sin t$, $\varphi_0(x_1) = x_1^3$, $\varphi(x_3)= x_3^3$, $\psi(x_1-x_3)= (x_1-x_3)^3$, and $h(x_2)= x_2^3$.   Obviously, the initial data $t_0 = 0$, $x_0 =(0,0,0)^T$ satisfy the consistency condition $(t_0,x_0)\in L_0$, that is, $B P_2 x_0 = Q_2 f(t_0,x_0)$.  Below, it is shown how the plots of the components $x_1(t)=I_L(t)$, $x_2(t)=U_C(t)$ and $x_3(t)=I(t)$ of the solution $x(t) =(x_1(t),x_2(t),x_3(t))^T$ obtained by methods~1~and~2 change with the mesh refinement ($h=0.1,\,0.01,\,0.001,\,0.0001$).
\begin{figure}[H]%
\centering
\includegraphics[width=6.8cm]{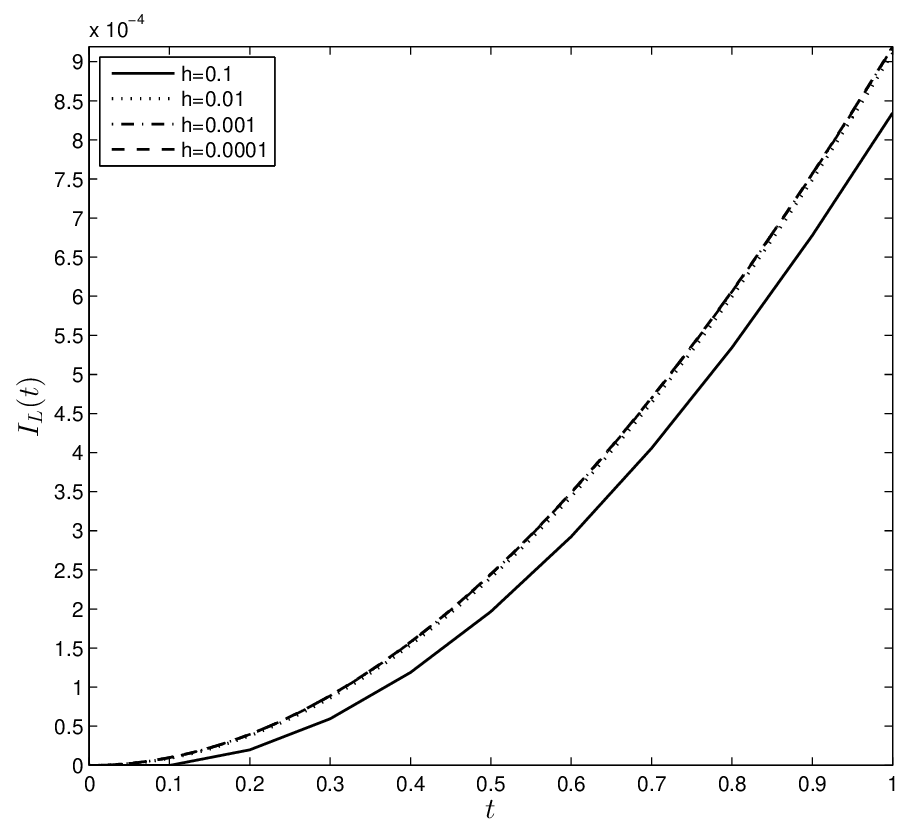}
\includegraphics[width=7cm]{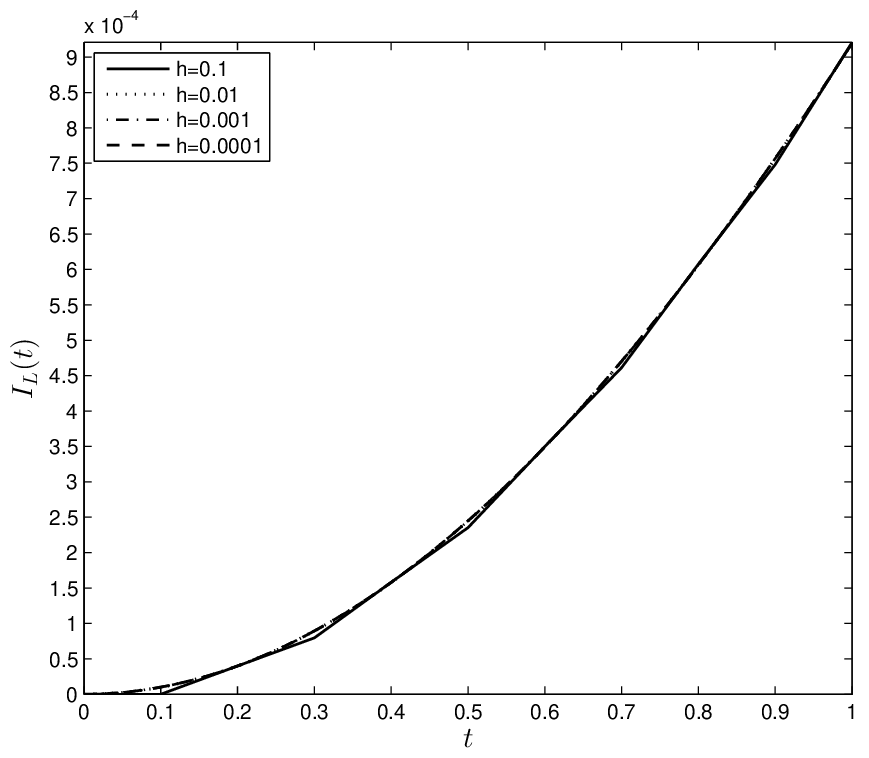}
\caption{The changes of the component $I_L(t)$ of the solution obtained by method~1~(left) and method~2~(right) when changing the step size $h=0.1,\,0.01,\,0.001,\,0.0001$}\label{modSinCom1}
\end{figure}
\begin{figure}[H]%
\centering
\includegraphics[width=6.8cm]{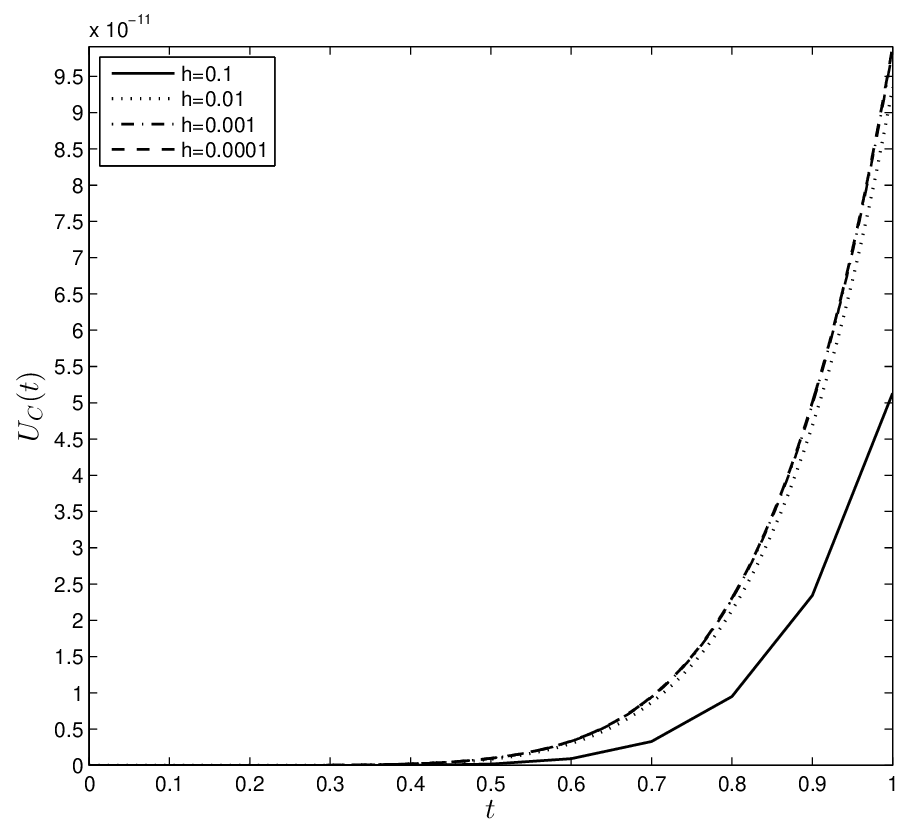}
\includegraphics[width=7cm]{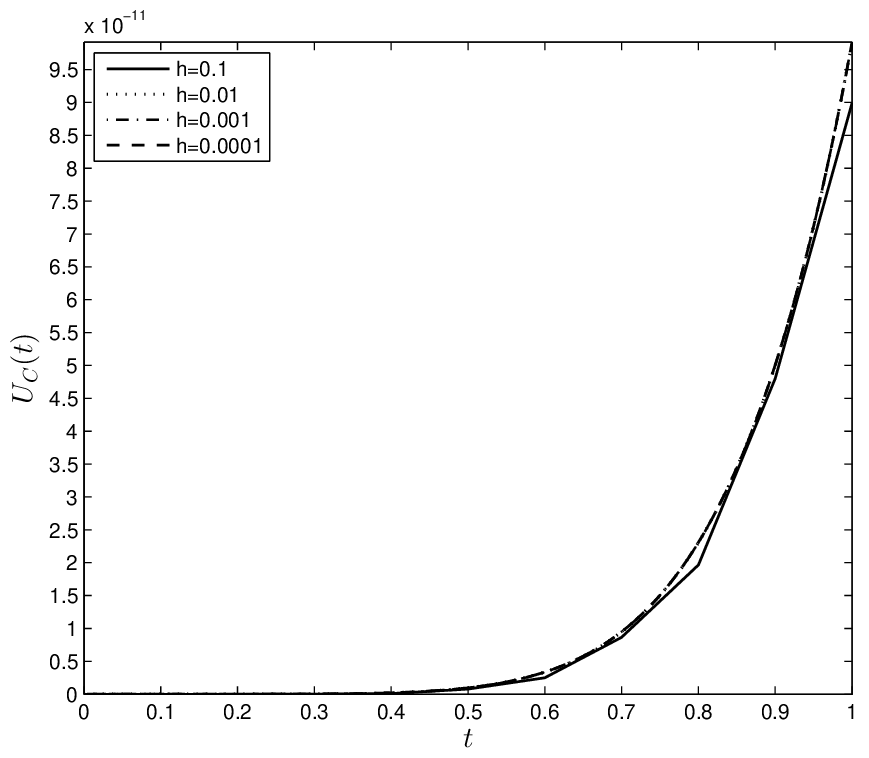}
\caption{The changes of the component $U_C(t)$ of the solution obtained by method~1~(left) and method~2~(right) when changing the step size $h=0.1,\,0.01,\,0.001,\,0.0001$}\label{modSinCom2}
\end{figure}
\begin{figure}[H]%
\centering
\includegraphics[width=6.8cm]{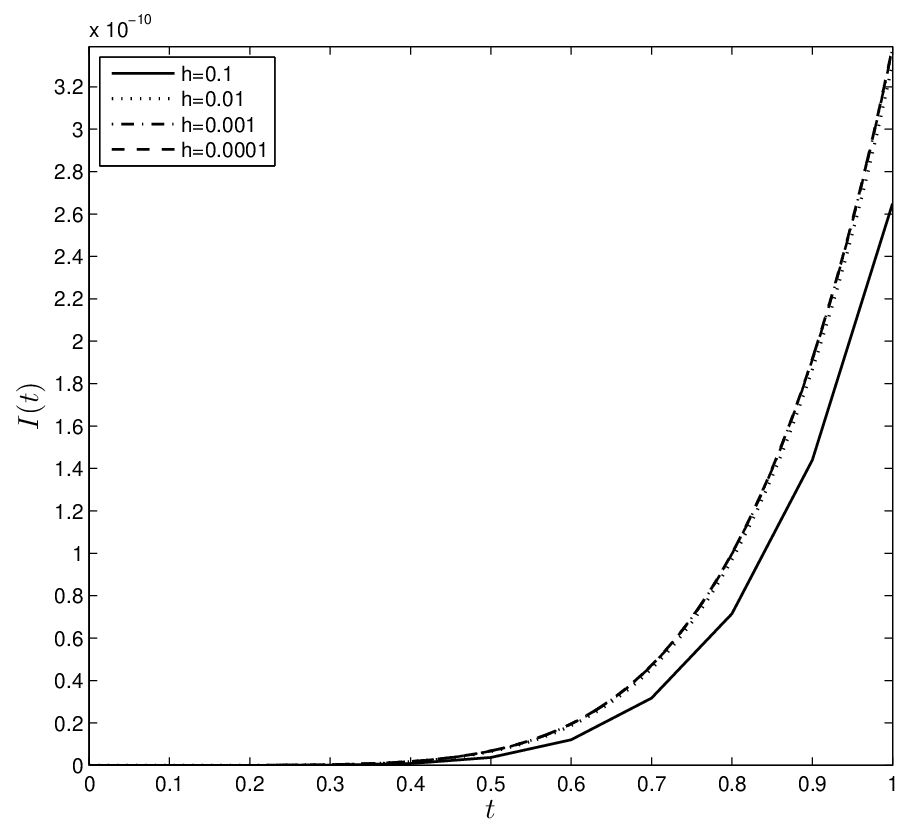}
\includegraphics[width=7cm]{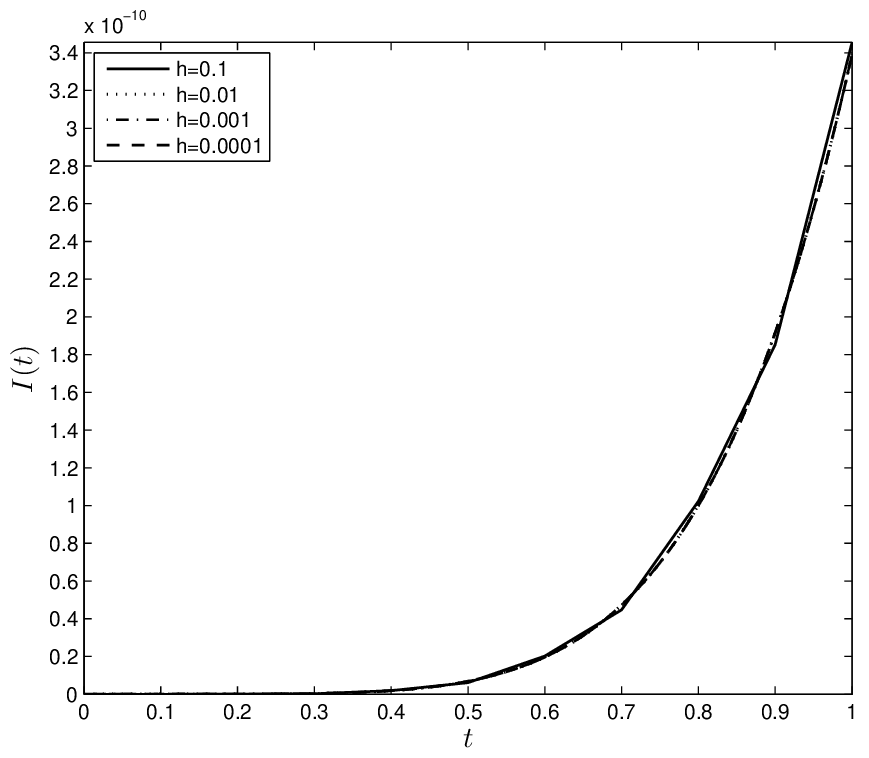}
\caption{The changes of the component $I(t)$ of the solution obtained by method~1~(left) and method~2~(right) when changing the step size $h=0.1,\,0.01,\,0.001,\,0.0001$}\label{modSinCom3}
\end{figure}
For the solution components $I_L(t)$ and $U_C(t)$ shown in Figures~\ref{modSinCom1} and~\ref{modSinCom2}, the values at the points ${t=0.2,\,0.4,\,0.6,\,0.8,\,1}$ are presented below in Tables~1 and~2.
\begin{center}
\textbf{Table 1}\; The approximate values of $I_L(t)$ at the points $t=0.2,\,0.4,\,0.6,\,0.8,\,1$\medskip

\noindent\renewcommand{\tabcolsep}{0.1cm}\begin{tabular}{|c|c|c|c|c|c|c|}%
 \hline
   & \multicolumn{2}{c|}{$I_L(0.2)$} & \multicolumn{2}{c|}{$I_L(0.4)$} & \multicolumn{2}{c|}{$I_L(0.6)$} \\
 \cline{2-7} \rule{0cm}{0.4cm}
   $h$ & Method 1 & Method 2 & Method 1 & Method 2 & Method 1 & Method 2 \\
 \hline \rule{0cm}{0.4cm}
     $10^{-1}$   &  1.9967e-05  & 3.9933e-05 & 1.1880e-04  & 1.5814e-04 &  2.9257e-04  &  3.4991e-04 \\ 
 \hline \rule{0cm}{0.4cm}
     $10^{-2}$   &  3.7880e-05 & 3.9868e-05  & 1.5398e-04 & 1.5788e-04 & 3.4368e-04 &  3.4933e-04 \\ 
 \hline \rule{0cm}{0.4cm}
     $10^{-3}$  &  3.9668e-05 &  3.9867e-05 & 1.5749e-04 & 1.5788e-04 & 3.4876e-04 &  3.4933e-04 \\ 
 \hline \rule{0cm}{0.4cm}
     $10^{-4}$ &  3.9847e-05 &  3.9867e-05 & 1.5784e-04 & 1.5788e-04 & 3.4927e-04 & 3.4933e-04  \\  
 \hline
 \end{tabular}
\end{center}
\begin{center}
\noindent \renewcommand{\tabcolsep}{0.1cm}\begin{tabular}{|c|c|c|c|c|}%
 \hline
   & \multicolumn{2}{c|}{$I_L(0.8)$} & \multicolumn{2}{c|}{$I_L(1)$}\\
 \cline{2-5} \rule{0cm}{0.4cm}
    $h$ & Method 1 & Method 2 & Method 1 & Method 2 \\
 \hline \rule{0cm}{0.4cm}
    $10^{-1}$    &  5.3435e-04  & 6.0760e-04 & 8.3448e-04  & 9.2093e-04  \\
 \hline \rule{0cm}{0.4cm}
    $10^{-2}$  &  5.9941e-04 & 6.0660e-04  & 9.1097e-04 & 9.1941e-04 \\
 \hline \rule{0cm}{0.4cm}
    $10^{-3}$  & 6.0587e-04 & 6.0659e-04  & 9.1855e-04 & 9.1940e-04 \\
 \hline \rule{0cm}{0.4cm}
    $10^{-4}$ & 6.0651e-04  & 6.0659e-04  & 9.1931e-04 & 9.1940e-04 \\
 \hline
 \end{tabular}
\end{center}
\begin{center}
\textbf{Table 2}\; The approximate values of $U_C(t)$ at the points $t=0.2,\,0.4,\,0.6,\,0.8,\,1$\medskip

\noindent \renewcommand{\tabcolsep}{0.1cm}\begin{tabular}{|c|c|c|c|c|c|c|}%
 \hline
   & \multicolumn{2}{c|}{$U_C(0.2)$} & \multicolumn{2}{c|}{$U_C(0.4)$} & \multicolumn{2}{c|}{$U_C(0.6)$} \\
 \cline{2-7} \rule{0cm}{0.4cm}
   $h$ & Method 1 & Method 2 & Method 1 & Method 2  & Method 1 & Method 2 \\
 \hline \rule{0cm}{0.4cm}
     $10^{-1}$ & 0 & 0 & 2.1963e-14 & 9.6804e-14 & 9.2137e-13 & 2.4827e-12 \\
 \hline \rule{0cm}{0.4cm}
     $10^{-2}$ & 1.2255e-15  & 1.7053e-15  & 1.7884e-13 & 2.1045e-13 & 3.0209e-12 &  3.3564e-12  \\
 \hline \rule{0cm}{0.4cm}
     $10^{-3}$ &  1.6937e-15 & 1.7522e-15  & 2.0837e-13 & 2.1184e-13 & 3.3303e-12 &  3.3659e-12  \\
 \hline \rule{0cm}{0.4cm}
     $10^{-4}$ & 1.7468e-15  & 1.7527e-15 & 2.1150e-13 & 2.1185e-13 & 3.3624e-12 &  3.3660e-12 \\
 \hline
 \end{tabular}
\end{center}
\begin{center}
\noindent \renewcommand{\tabcolsep}{0.1cm}\begin{tabular}{|c|c|c|c|c|}%
 \hline
   & \multicolumn{2}{c|}{$U_C(0.8)$} & \multicolumn{2}{c|}{$U_C(1)$}\\
 \cline{2-5} \rule{0cm}{0.4cm}
    $h$ & Method 1 & Method 2 & Method 1 & Method 2 \\
 \hline \rule{0cm}{0.4cm}
    $10^{-1}$   & 9.5030e-12 & 1.9667e-11 & 5.1291e-11 & 8.9939e-11  \\
 \hline \rule{0cm}{0.4cm}
    $10^{-2}$   & 2.1361e-11  & 2.3049e-11 & 9.3469e-11 & 9.9068e-11 \\
 \hline \rule{0cm}{0.4cm}
    $10^{-3}$ &  2.2908e-11 & 2.3084e-11  & 9.8584e-11 & 9.9162e-11 \\
 \hline \rule{0cm}{0.4cm}
    $10^{-4}$ & 2.3067e-11  &  2.3085e-11 & 9.9105e-11 & 9.9163e-11 \\
 \hline
 \end{tabular}
 \end{center}
Figures~\ref{modSinCom1}--\ref{modSinCom3} and Tables~1,~2 show that method~2 (the method \eqref{Impmet1}--\eqref{Impmet4})  converges faster to the exact solution than method~1 (the method \eqref{met1}--\eqref{met4}).

However, when choosing a more suitable method for solving a certain problem,  one should also take into account the quantitative characteristics of stability of the methods, namely, the values of the coefficients $h M_1$, $g(h)=\|I-h G^{-1} B\| +h M_1\le 1+h(\|G^{-1} B\| +M_1)$ in \eqref{num8} (for method~1) and $2h M_1$, $\|\Hat g(h)\|=1+2h(\|G^{-1} B\| +M_1)$ in \eqref{num8new} (for method~2) and the length of the interval $[t_0,T]$ on which the computation is performed.   The larger the quantitative characteristic of stability of a method (the quantities characterizing the stability of the method), the smaller step size should be chosen to achieve the required accuracy of computations. In turn, the choice of a smaller step size leads to larger roundoff errors.   In this connection, to carry out a computation with a given accuracy using method~2, it may be necessary to choose a much smaller step size and consequently to spend much more time than using method~1.  The differences in the mentioned coefficients from \eqref{num8} and \eqref{num8new} arise because the derivative is approximated by a forward difference, obtained from the representation \eqref{Taylor1}, in method~1 and by a centered difference, obtained from \eqref{Taylor3}, in method~2.  Similar differences arising when using such approximations for the linear differential equation $\dot{x}+Ax=f(t)$ are described in detail in [\cite{GodRyab}, Chapter~5].  It is clear that for small values of $\|G^{-1} B\|$, $M_1$ or a sufficiently small interval $[t_0,T]$ there are no distinct differences in stability.

In particular, for the considered DAE \eqref{DAE}, \eqref{two-pole}, the larger $L$, $C$, $r$ and the smaller $g$, $M_1$, the less differences in stability.   Also, the differences in stability are small on a sufficiently small time interval $[t_0,T]$ and Figures~\ref{modSinCom1}--\ref{modSinCom3} confirm this.  If we compute the solution for the above quantities and functions on a larger time interval, then method~2 starts to lose stability.   This is not visible in Figure~\ref{mSinC1}, but this is already noticeable in Figure~\ref{mSinC2-3}.  Therefore, it is necessary to decrease the step size, as shown in Figure~\ref{mSinC4h2-3}.
\begin{figure}[H]%
\centering
\includegraphics[width=6.6cm]{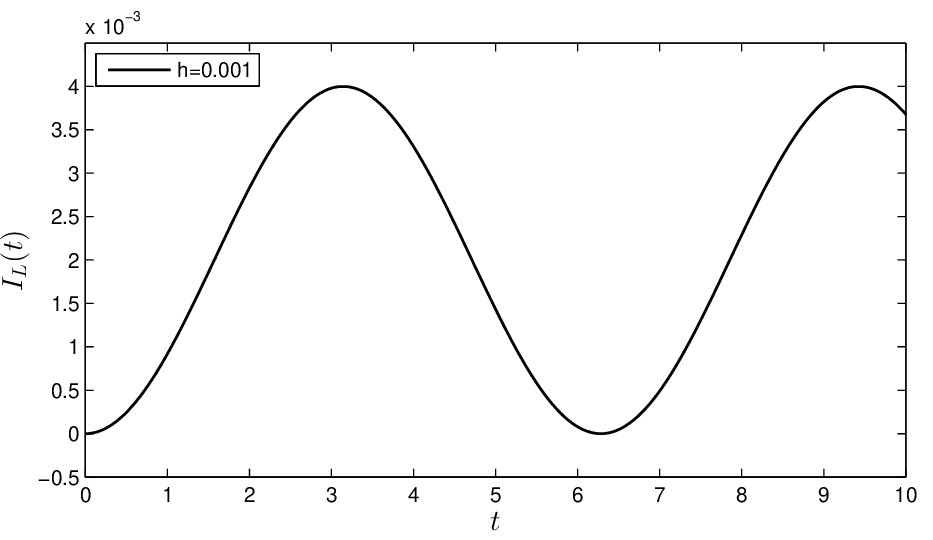}
\includegraphics[width=6.6cm]{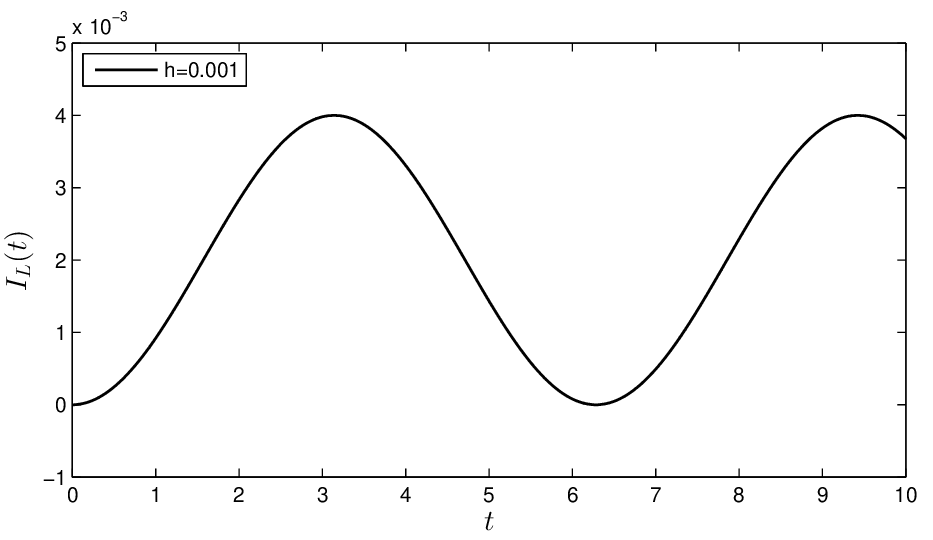}
\caption{The component $I_L(t)$ of the solution obtained by method~1 (left) and method~2 (right); the step size $h=0.001$}\label{mSinC1}
\end{figure}
\begin{figure}[H]%
\centering
\includegraphics[width=6.6cm]{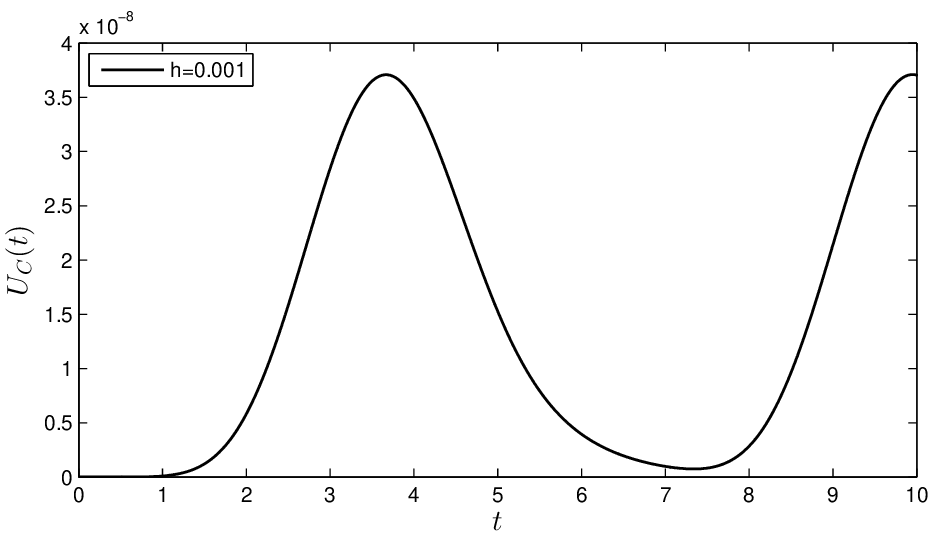}
\includegraphics[width=6.6cm]{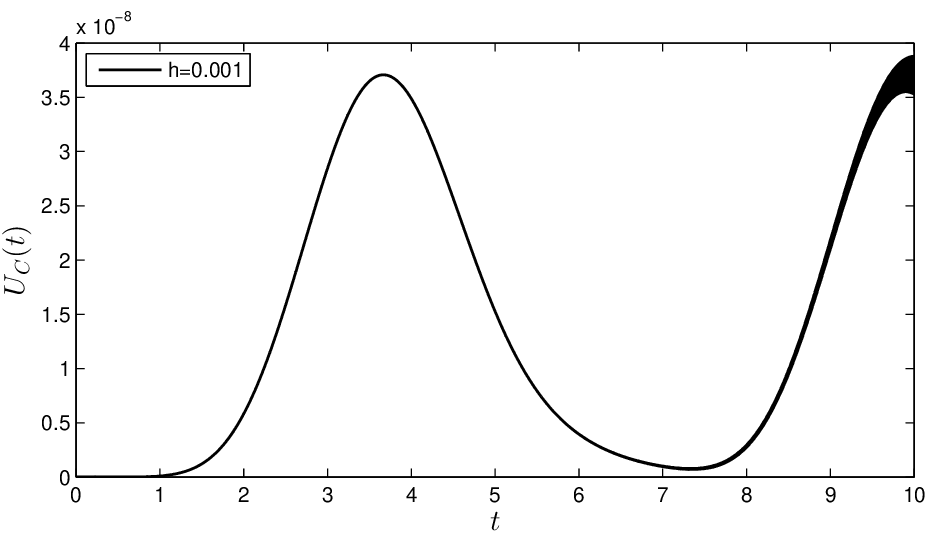}

\includegraphics[width=6.6cm]{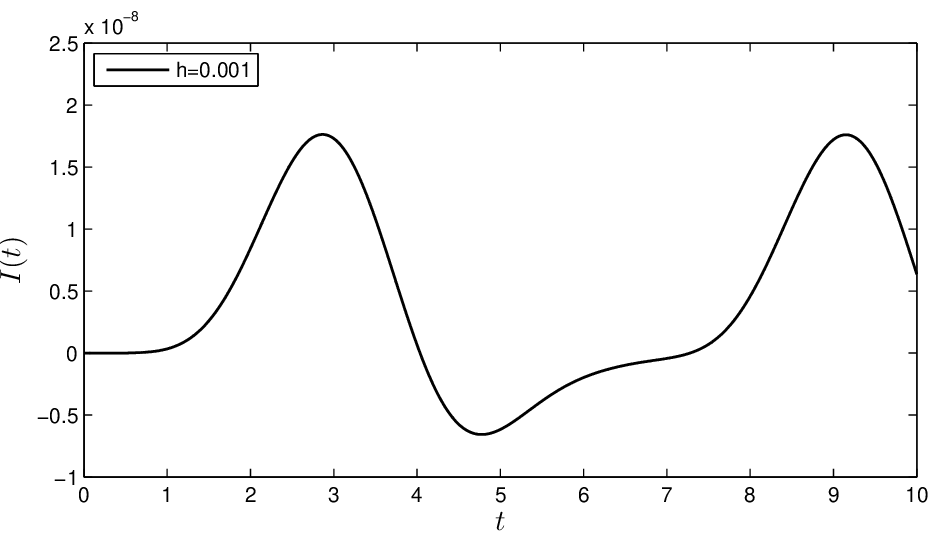}
\includegraphics[width=6.6cm]{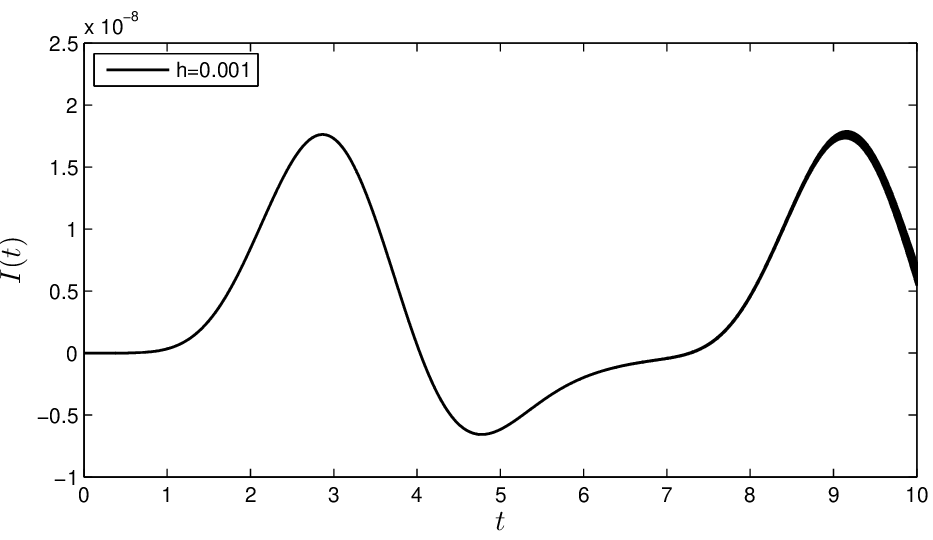}
\caption{The components $U_C(t)$, $I(t)$ of the solution obtained by method~1 (left) and method~2 (right); the step size $h=0.001$}\label{mSinC2-3}
\end{figure}
\begin{figure}[H]%
\centering
\includegraphics[width=6.5cm]{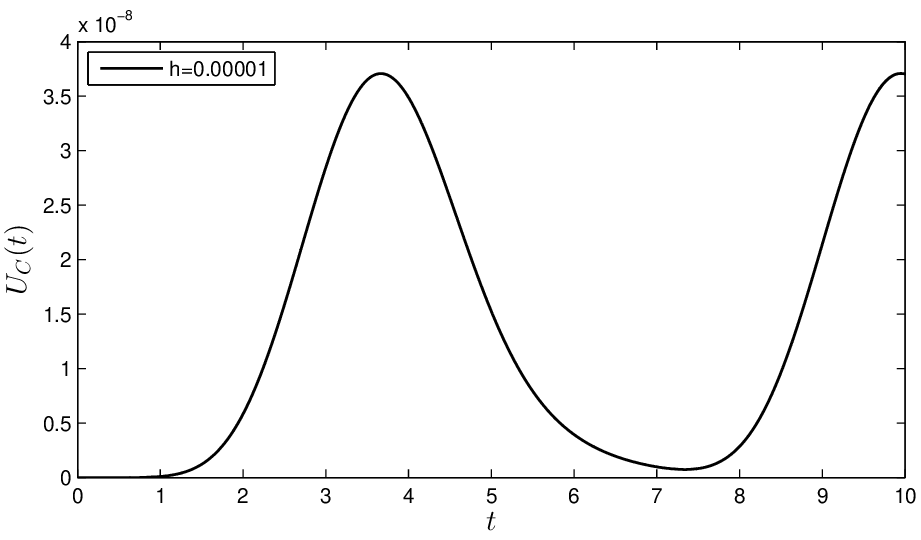}
\includegraphics[width=6.5cm]{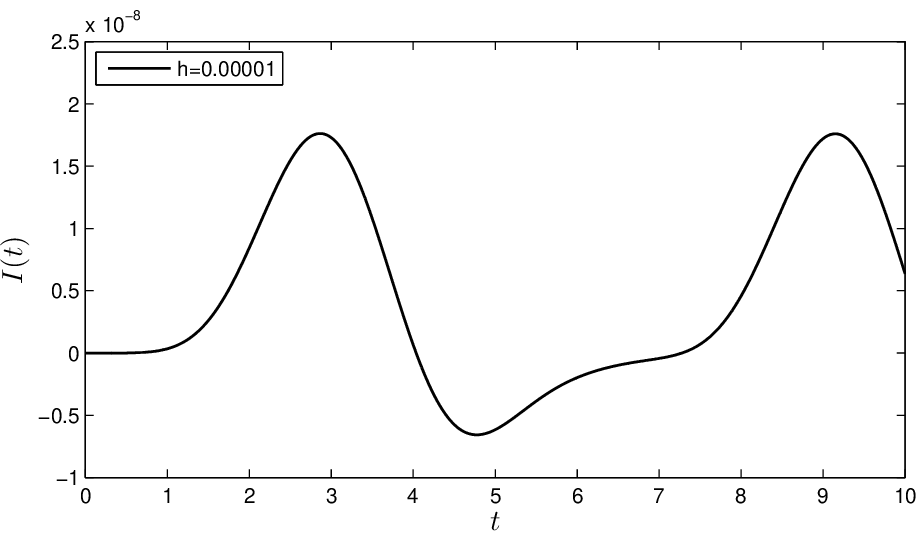}
\caption{The components $U_C(t)$, $I(t)$ obtained by method~2; the step size $h=0.00001$}\label{mSinC4h2-3}
\end{figure}
However, for the larger value of $r=4$ (the value is doubled) and the smaller value of $g=0.1$ (the value is decreased by 2 times), method~2 is stable for the same step size and time interval that were originally (see Figure~\ref{mSinC2new}) and, accordingly, there are no visible differences in the stability of the methods.
\begin{figure}[H]%
\centering
\includegraphics[width=6.5cm]{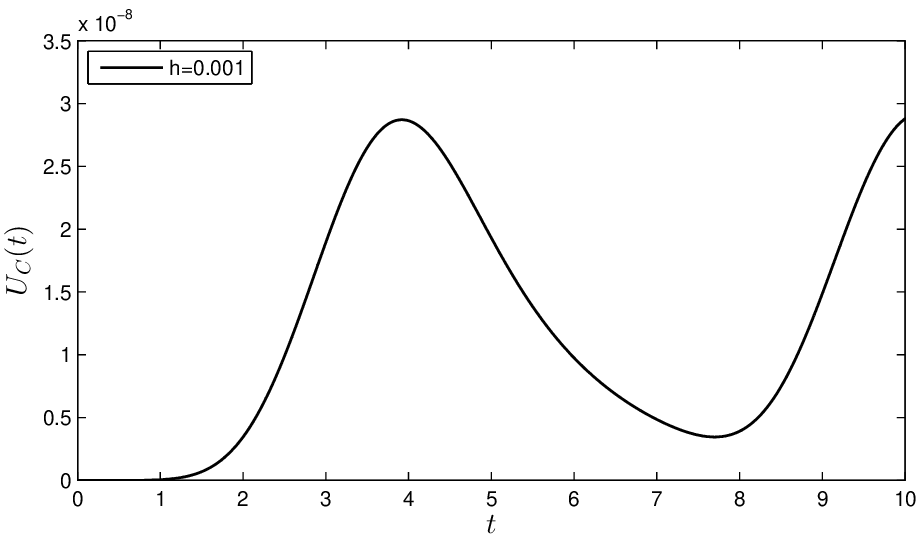}
\includegraphics[width=6.5cm]{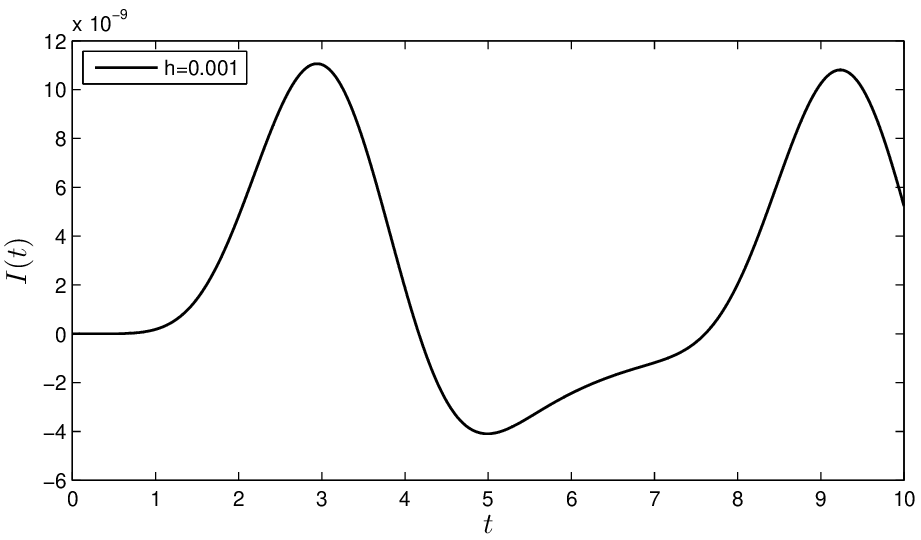}
\caption{The components $U_C(t)$, $I(t)$ obtained by method~2 for the new $r$, $g$ and the original step size $h=0.001$}\label{mSinC2new}
\end{figure}

\emph{We can draw the following conclusions about the effectiveness of application of the methods in various situations}.   When the parameters in the DAE~\eqref{DAE} are chosen arbitrarily, method~2 is better for computing or refining a solution on sufficiently small time intervals (or on the parts of a large interval),  since it has a higher order of accuracy with respect to the step size $h$.   For large time intervals, it is better to use method~1, since it is  more stable (has a smaller coefficient characterizing stability). However, once again note that method~2 can be more effective on a large time interval if the parameters, which are chosen on the basis of the conditions of the problem being solved, do not have a large negative effect on the stability.

\section{Conclusions and outlooks}

To solve semilinear DAEs, the combined methods of the first and second orders of accuracy were obtained in this work. The correctness and convergence of the methods were proved.  The developed methods enable to find a solution on any given time interval  and require weaker restrictions for the nonlinear part of the equation than other known methods.  Also, the effectiveness of the developed methods is due to the possibility to numerically find the spectral projectors, using the formulas \eqref{ProjRes}, which enables to numerically solve and analyze a semilinear DAE in the original form without additional analytical transformations.  In Section~\ref{Appl}, the numerical analysis of the semilinear DAE describing the mathematical model for a nonlinear electrical circuit has been carried out.  The results of the numerical analysis verify the results of the theoretical studies of global dynamics for the mathematical model.   The plots of numerical solutions, demonstrating the evolutionary properties of the considered model, have been presented.

In the future, it is planned to improve the methods obtained in this paper and to extend them to a certain class of semilinear DAEs with a singular operator pencil (the theorem on the unique global solvability and the Lagrange stability for this class of equations was proved in \cite{Fil.UMJ}).

\section*{Acknowledgement}
 The publication is based on the research provided by grant support of the State Fund for Fundamental Research of Ukraine (project No.~F83/82-2018), and the National Academy of Sciences of Ukraine (project ``Qualitative, asymptotic and numerical analysis of various classes of differential equations and dynamical systems, their classification, and practical application'', state registration number~0119U102376).




\end{document}